\documentclass{amsart}

\usepackage{euscript}
\usepackage{amsmath}
\usepackage{amscd}
\usepackage{amssymb}
\usepackage{enumerate}
\usepackage{amsfonts}
\usepackage{setspace}
\usepackage{eufrak}
\usepackage{array}
\usepackage{bbold}
\usepackage{tikz}
\usetikzlibrary{matrix,arrows,calc,intersections,fit}
\usepackage{tikz-cd}
\usepackage{pgfplots}
\usepackage{textcomp}

\newcommand{\bC}{{\mathbb C}}
\newcommand{\bP}{{\mathbb P}}
\newcommand{\bR}{{\mathbb R}}

\newcommand{\bN}{{\mathbb N}}

\newcommand{\bT}{{\mathbb T}}

\newcommand{\bZ}{{\mathbb Z}}

\newcommand{\bk}{{\mathbf k}}

\newcommand{\cF}{\mathcal F}

\newcommand{\cM}{\mathcal M}

\newcommand{\cY}{\mathcal Y}

\newcommand{\scrL}{\EuScript L}

\newcommand{\scrO}{\EuScript O}

\newcommand{\Def}{\operatorname{Def}}

\newcommand{\bfk}{\mathbf{k}}

\newcommand{\id}{\operatorname{id}}

\renewcommand{\mod}{\operatorname{mod}}

\newcommand{\Coh}{\operatorname{Coh}}

\newcommand{\Ext}{\operatorname{Ext}}

\newcommand{\Sym}{\operatorname{Sym}}

\newcommand{\Spin}{\operatorname{Spin}}

\newcommand{\dbar}{\overline{\partial}}

\renewcommand{\max}{\operatorname{max}}

\def\co{\colon\thinspace}

\numberwithin{equation}{subsection}

\newtheorem{thm}{Theorem}[section]
\newtheorem{cor}[thm]{Corollary}
\newtheorem{lem}[thm]{Lemma}
\newtheorem{lemma}[thm]{Lemma}
\newtheorem{prop}[thm]{Proposition}
\newtheorem{defin}[thm]{Definition}
\newtheorem{def-lem}[thm]{Definition-Lemma}

\theoremstyle{remark}

\newtheorem{rem}[thm]{Remark}

\newtheorem{example}[thm]{Example}

\newcommand{\superscript}[1]{\ensuremath{^{\textrm{#1}}} }

\renewcommand{\th}[0]{\superscript{th}}
\newcommand{\st}[0]{\superscript{st}}

\newcommand{\comment}[1]{}
\DeclareMathOperator*{\colim}{colim}

\title[Local HMS]{Homological Mirror Symmetry for local SYZ singularities}

\author[M.~Abouzaid]{Mohammed Abouzaid}
\author[Z.~Sylvan]{Zack Sylvan}
\date{\today} 
\thanks{The first author was supported by NSF grant DMS-1308179,  DMS-1609148, and DMS-1564172, as well as the Simons Foundation through its ``Homological Mirror Symmetry'' Collaboration grant.}
\begin{document}

\begin{abstract}
Gross and Siebert identified a class of singular Lagrangian torus fibrations which arise when smoothing toroidal degenerations, and which come in pairs that are related by mirror symmetry. We identify an immersed Lagrangian in each of these local models which supports a moduli space of objects that is isomorphic to the mirror space, and prove a homological mirror statement along the way.  
\end{abstract}

\maketitle
\tableofcontents
\section{Introduction}

For each pair $(m,n)$ of integers, consider the smooth hypersurface
\begin{equation}
X_{m,n} \equiv \{ (x_0, \ldots, x_m, y_1, \ldots, y_n) \in \bC^{m+1} \times \left( \bC^* \right)^{n} | \prod x_i = 1 + \sum y_j \},   
\end{equation}
which we equip with the restriction of the standard K\"ahler form on $\bC^{m+1} \times \left( \bC^* \right)^{n} $.
\begin{example}
For $m=0$, we obtain $(\bC^*)^n$. For $m=1$, we obtain the conic bundle
  \begin{equation}
    x_0 x_1 = 1 + \sum y_i.    
  \end{equation}
   For $n=1$, we obtain the hypersurface $\prod x_i = 1 + y $ in $ \bC^m \times \bC^*$, which is isomorphic to the complement in $\bC^m$ of the hypersurface $\prod x_i =1$.
\end{example}
It is well-known that $\left( \bC^* \right)^{n} $ is self-mirror, hence in our notation above that $X_{0,n}$ is mirror to $X_{m,0}$. The case $n=1$ or $m=1$ was extensively studied in \cite{AbouzaidAurouxKatzarkov2016}, where it was argued that a version of SYZ mirror symmetry holds for these pairs. There are by now several proofs of Homological mirror symmetry of other special cases, starting with the one given by Seidel for the Fukaya category of $X_{1,1}$ \cite{Seidel2013}, and including the one by Chan, Pomerleano and Ueda \cite{ChanPomerleanoUeda2016} for $X_{1,2}$, and recently by Pomerleano for $X_{1,n}$ \cite{Pomerleano2021} (there is also a literature on the sheaf theoretic side, which is reviewed in \cite{Gammage2021}, where the next result is also independently proved). The first result of this paper is an extension to the general case:
\begin{thm} \label{thm:HMS}
The wrapped Fukaya category of $X_{m,n}$ is equivalent to the (derived) category of coherent sheaves on $X_{n,m}$.
\end{thm}
Since $X_{n,m}$ is affine, the category of coherent sheaves is equivalent to the category of modules over the corresponding commutative ring. This means that we expect the existence of a Lagrangian in $X_{m,n}$ whose wrapped Floer cohomology is isomorphic to this ring; this turns out to be the positive real locus. The above theorem thus follows from a computation of the wrapped Floer cohomology of this Lagrangian, and a proof that it generates the Fukaya category.

\begin{rem}
  With some modification, we can extend our setup to the case where  $m$ is allowed to equal $-1$, which would then include the statement that the Fukaya category of the hypersurface $1 + \sum y_i = 0$ in $(\bC^*)^n$ (i.e. Mikhalkin's pair of pants \cite{Mikhalkin2004}) is equivalent to the derived category of coherent sheaves on the hypersurface $\prod x_i = 0$ in $\bC^n$ (see \cite{LekiliPolishchuk2020}). We shall in fact use this result as a building block of our proof.   
\end{rem}

The main goal of this paper is to obtain an understanding of mirror symmetry for these manifolds simultaneously at categorical (HMS) and the geometric (SYZ) levels: we thus seek to realise all skyscraper sheaves of points on one side as arising from (immersed) Lagrangians on the other side equipped with bounding cochains.

Our starting point is a variant of mirror symmetry for pairs of pants in dimension $1$ (c.f. \cite{AbouzaidAurouxEfimovKatzarkovOrlov2013}): the hypersurface
\begin{equation}
  H_{1} \equiv \{ 1 + y_1 + y_2 = 0\} \subset \left(\bC^*\right)^2  
\end{equation}
is mirror to the union of the coordinate lines in $\bC^2$.
Following Seidel's work \cite{Seidel2011}, we consider an immersed Lagrangian $L_1$ in the pair of pants, shown in Figure \ref{fig:Seidel-Lagrangian} as the complement of the origin and the point $1$ in $\bC$.
\begin{figure}
  \centering
  \begin{tikzpicture}
    \begin{scope}
      \node[label = below:{$0$}] (0) at (0,0) {$\ast$};
      \node[label = below:{$1$}] (1) at (2,0) {$\ast$};
      \node[label = left:{$s$}] (s) at (-1,0) {};
      
      \node[label = below:{$t$}] (t) at (1,0) {};
      \node[label = right:{$c$}] (c) at (3,0) {};
      \draw[thick] (3,0) .. controls (3.25,1) and (2,2) .. (1,2) .. controls (0,2) and (-1.25,1) .. (-1,0) .. controls (-.75,-1) and (.75, -1) ..  (1,0) .. controls (1.25,1)  and (2.75,1) .. (3,0);
  \end{scope}
  \begin{scope}[yscale = -1] 
    \draw[thick] (3,0) .. controls (3.25,1) and (2,2) .. (1,2) .. controls (0,2) and (-1.25,1) .. (-1,0) .. controls (-.75,-1) and (.75, -1) ..  (1,0) .. controls (1.25,1)  and (2.75,1) .. (3,0);
  \end{scope}
  \draw [fill=blue] (s) circle (0.1);
  \draw [fill=blue] (t) circle (0.1);
  \draw [fill=orange] (c) circle (0.1);
\end{tikzpicture}
  \caption{Seidel's immersed Lagrangian}
  \label{fig:Seidel-Lagrangian}
\end{figure}
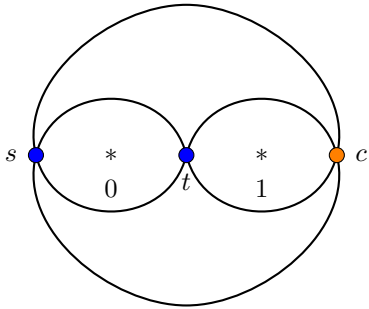
As we discuss in Section \ref{sec:mirrors-points}, by identifying an open subset of $X_{m,n}$ which is biholomorphic to $ (\bC^{*})^{m} \times H^n_1 $  we obtain from Seidel's construction a Lagrangian
\begin{equation}
  L_{m,n} \cong T^m \times L_1^n \subset   (\bC^{*})^{m} \times H^n_1 \subset X_{m,n}.
\end{equation}

Our main result is the following:
\begin{thm} \label{thm:Family-HMS}
  Every skyscraper sheaf of a point in $X_{n,m}$ is mirror to an object of the Fukaya category of $X_{m,n}$ supported on $L_{m,n}$.
\end{thm}

\begin{rem}
The construction of $L_{m,n}$ is not unique. As the reader may see by consulting Equation \eqref{eq:projection_qj}, the embedding of $ (\bC^{*})^{m} \times H^n_1 $ that we use depends on an ordering of the coordinates. One can construct a more symmetric immersed Lagrangian as follows:  consider the projection $\pi \co X_{m,n} \to \bC$ given by $\prod_{i=0}^{m} x_i$, which equals $1 + \sum_{j=1}^{n} y_j $ by definition.  The fibre at a point $\lambda$ is the product of the hypersurfaces $\prod_{i=0}^{m} x_i = \lambda$ and $\sum_{j=1}^{n} y_j  = \lambda -1 $.  Since both function are equivariant, we see that the only critical points are $0$ and $1$, and that the fibres away from these points are equivalent to the product of a torus $\left(\bC^{*}\right)^m$ with an $n-1$ dimensional pair of pants.   In \cite{Sheridan2011}, Sheridan constructed an immersed Lagrangian in $H_{n-1}$, which contains all the information about the mirror. We expect that it is possible to prove an analogue of Theorem \ref{thm:Family-HMS} for the Lagrangian obtained by parallel transport of the product of $T^m$ with Sheridan's Lagrangians, along the curve in the base given by Figure \ref{fig:Seidel-Lagrangian}. 
\end{rem}

The most delicate part of Theorem \ref{thm:Family-HMS} is to make sense of the Floer theory of $L_{m,n}$, because it bounds holomorphic discs. The space of objects of the Fukaya category supported on an immersed Lagrangian arises as the zero locus of a function (the Maurer-Cartan function) on the \emph{deformation space}, which can be described from the product of the space of local systems (a copy of $\left(\bk^{*}\right)^{m+n}$), together with a coordinates for each self-intersection of degree $1$. The constraint on the Maurer-Cartan function enforces the vanishing of the weighted count of holomorphic discs with one output.

Instead of studying the full deformation space, which is quite large in our situation, we will consider a distinguished subspace which is a product of $\left(\bk^{*}\right)^{m} $ (corresponding to the $T^m$ factors) with $ \bk^{2n}$ (corresponding to the self intersection points $s$ and $t$ shown in Figure \ref{fig:Seidel-Lagrangian}). It will turn out that the Maurer-Cartan function vanishes identically on this subspace, but we will see the equation of the mirror arise from the equation required for the resulting object to be non-trivial. A difficulty is caused by the fact that these objects correspond to a direct sum of $2^{n}$ copies of skyscraper sheaves, requiring us to pass to the idempotent closure of the Fukaya category in order to obtain the desired correspondence. 
\begin{rem} \label{rem:discussion_mirror_tori}
  It is illuminating to consider the mirror correspondence for some other Lagrangians as well. Let $L_1^{s}$ and $ L_1^{t}$ denote either component of the Lagrangians shown in Figure \ref{fig:Chekanov-Clifford-Lagrangian} (more precisely, we should take the exact representatives of these Lagrangians), which are obtained from $L_1$ by surgery at the intersection points labelled $s$ and $t$. For each $0 \leq i \leq n$, we can then consider the Lagrangians
  \begin{equation}
    L_{m,n}^{i} \equiv T^{m} \times (L_1^s)^{i} \times (L_1^t)^{n-i}.
  \end{equation}
  The analogous result to Theorem \ref{thm:Family-HMS} is that each skyscraper sheaf of a point of $X_{m,n}$ with the property that the coordinates $x_k$ for $k \neq i$ do not vanish, admits a mirror which is supported on $L_{m,n}^i$ (c.f. Remark \ref{rem:proof_mirror_tori}). 

One can more directly describe the Hamiltonian isotopy class of these Lagrangians as follows: assume that $i \neq 0$, and consider a torus in $H_{n-1}$ with $1 \sim |y_i|$, and $|y_k| \ll 1$ for $k \neq i$ (more precisely, under a tropical deformation of the equation for $H_{n-1}$), which we may assume to be an exact Lagrangian by an appropriate deformation for the primitive. The parallel transport of the product of this Lagrangian with the real torus in $(\bC^*)^{m}$ along the curve $ L_1^t $ in the base corresponds to $L_{m,n}^i $.  For $i=0$, one can use the fact that $\left(\bC^*\right)^{n}$ retracts to the fibre of $1 + \sum y_i$ over a disc bounding $L_1^s$, to see that the parallel transports of these Lagrangians along $L_1^s$ are Hamiltonian isotopic, and correspond to the Lagrangian $L_{m,n}^{0}$.

\end{rem}

\begin{figure}
  \centering
  \begin{tikzpicture}
    \begin{scope}[xscale = -1]
      \node[label = right:{$1$}] (0) at (0,0) {$\ast$};
      \node[label = right:{$0$}] (1) at (2,0) {$\ast$};
     \node[label = above:{$L_1^s$}] (s) at (2,1.5) {};
      \draw[thick] (2.75,0.5) .. controls (3.25,0.5) and (2,2) .. (1,2) .. controls (0,2) and (-1.25,1) .. (-1,0) .. controls (-.75,-1) and (.75, -1) ..  (1,0) .. controls (1.25,1)  and (2.5,.5) .. (2.75,0.5);
      \begin{scope}[yscale = -1] 
    \draw[thick] (2.75,0.5) .. controls (3.25,0.5) and (2,2) .. (1,2) .. controls (0,2) and (-1.25,1) .. (-1,0) .. controls (-.75,-1) and (.75, -1) ..  (1,0) .. controls (1.25,1)  and (2.5,.5) .. (2.75,0.5);
  \end{scope}
        \end{scope}
     
        \begin{scope}[shift={(4,0)}]
            \node[label = right:{$0$}] (0) at (0,0) {$\ast$};
            \node[label = right:{$1$}] (1) at (2,0) {$\ast$};
                 \node[label = above:{$L_1^t$}] (t) at (2,1.6) {};
      \draw[thick] (3,0) .. controls (3.25,1) and (2,2) .. (1,2) .. controls (0,2) and (-1.25,1) .. (-1,0) .. controls (-.75,-1) and (.75, -.25) ..  (1,-.25) .. controls (1.25,-.25)  and (2.75,-1) .. (3,0);
       \begin{scope}[yscale = -1] 
    \draw[thick] (3,0) .. controls (3.25,1) and (2,2) .. (1,2) .. controls (0,2) and (-1.25,1) .. (-1,0) .. controls (-.75,-1) and (.75, -.25) ..  (1,-.25) .. controls (1.25,-.25)  and (2.75,-1) .. (3,0);
  \end{scope}
  \end{scope}
\end{tikzpicture}
  \caption{The Lagrangians $L^s_1$ and $L^t_1$.}
  \label{fig:Chekanov-Clifford-Lagrangian}
\end{figure}
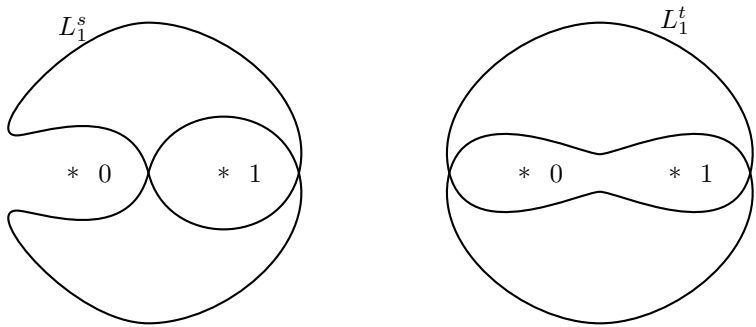

\subsection*{Acknowledgements}

We would like to thank Sasha Efimov and Alex Perry for help with the results of Section \ref{sec:gluing-coher-sheav}, and Yanki Lekilli for useful correspondence about the results of \cite{LekiliPolishchuk2020}.

\section{Mirrors to points}
\label{sec:mirrors-points}
In this section, we construct in each manifold $X_{m,n}$ a Lagrangian $L_{m,n}$, which will play a key role in our SYZ interpretation of mirror symmetry in the sense that we will identify a subspace of its space of non-trivial bounding cochains with a copy of the mirror space.

\subsection{Divisors and iterated fibrations on $X_{m,n}$}
\label{sec:iter-fibr-x_m}

Recall that
\begin{equation}
X_{m,n} \equiv \{ (x_0, \ldots, x_m, y_1, \ldots, y_n) \in \bC^{m+1} \times \left( \bC^* \right)^{n} | \prod x_i = 1 + \sum y_j \}.    
\end{equation}
We introduce a collection of maps
\begin{equation}
  q_j \co  X_{m,n} \to \bC 
\end{equation}
indexed by $1 \leq j < n$, and which are expressed in terms of the $\bC^*$-coordinates via the formula 
\begin{equation} \label{eq:projection_qj}
  q_j(x_0, \ldots, x_m, y_1, \ldots, y_n) = 1 + y_1/y_{j+1} + \cdots + y_{j}/y_{j+1}.
\end{equation}
We define $q_{n}$ to be the projection map to $\bC$ discussed in the introduction. We note that this last map factors through the projection map to $\bC^{m+1}$:
\begin{equation}
  \begin{tikzcd}
     X_{m,n}  \ar[r] & \bC^{m+1} \ar[r] & \bC.
  \end{tikzcd}
\end{equation}

We introduce the divisor $D_j \subset X_{m,n}$, given as the inverse image of $1$ under $q_j$:
\begin{equation}
  D_j \equiv q_j^{-1}(1).  
\end{equation}
\begin{lem}
  \label{lem:iterative_description_divisors}
  The divisor $D_j$ agrees with the inverse image of $0$ under $q_{j-1}$.
\end{lem}
\begin{proof}
The divisor $D_j$ is the solution to the equation $\sum_{j=1}^{j} y_j = 0$, which holds if and only the quotient by $y_j$ vanishes, which gives the $0$ locus of $q_{j-1}$.
\end{proof}
The above Lemma suggests introducing an additional divisor
\begin{equation}
  D_{n+1} \equiv q^{-1}_{n}(0),  
\end{equation}
which is the normal crossing divisor $\prod x_i = 0$.

Let $D$ denote the union of the divisors $\{D_j\}_{j=1}^{n+1}$. For the next result, it is convenient to recall that the function $q_n$ is the product of the coordinates $x_i$:
\begin{lem}
  \label{lem:product_decompositions}
  The maps $(x_1, \dotsc, x_m, q_1, \dotsc, q_n)$ define coordinates on  $X_{m,n} \setminus D$ identifying this space with a product of $(\bC^{*})^{m}$ with ($1$-dimensional) pairs of pants:
  \begin{equation}
    \begin{tikzcd}
      X_{m,n} \setminus D \ar[r,"\cong"] \ar[d] &  (\bC^{*})^{m} \times H^n_1  \ar[d] \\
      X_{m,n} \ar[r] & \bC^{m} \times \bC^{n}.
    \end{tikzcd}
  \end{equation}
  \qed
\end{lem}

This product decomposition allows us to construct our desired Lagrangian $L_{m,n}$ as a product
\begin{equation}
  L_{m,n} \cong T^m \times L_{1}^n,  
\end{equation}
where $L_1$ is Seidel's Lagrangian in the pair of pants, and $T^m \subset  (\bC^{*})^{m}$ is the locus where the norm of the coordinates is $1$. As in \cite{Seidel2011}, we equip $L_1$ (and the standard circle in $ \bC^{*}$) with the bounding $\Spin$ structure, which determines a $\Spin$ structure on $L_{m,n}$.

\begin{rem}
We have used the ordering of the coordinates to construct this splitting. The Lagrangian $L_{m,n}$ depends on this splitting. However, the outcome of our work will imply that the Maurer-Cartan spaces which we will associate to these Lagrangians are canonically isomorphic, because they can be identified with the spectrum of the commutative ring $SH^0(X_{m,n})$.
\end{rem}

\subsection{Symplectic and holomorphic volume forms}
\label{sec:sympl-holom-volume}

Since $X_{m,n}$ is an affine variety, it is equipped with a Liouville structure which is canonical up to symplectomorphism \cite{EliashbergGromov1991,SeidelSmith2005}, and is induced by pulling back the standard symplectic form in $\bC^{N}$ under a proper holomorphic embedding. On any open subset of $X_{m,n}$, we may arrange for the symplectic form to agree with a preferred choice, and we shall use this flexibility to ensure that it agrees with the product symplectic form on $ (\bC^{*})^{m} \times H^n_1 $ away from a neighbourhood of the divisor $D$. In this way, $L_{m,n}$ is evidently a Lagrangian submanifold, which is exact for an appropriate choice of primitive.

The natural holomorphic volume form on $X_{m,n}$ is given by
\begin{equation} \label{eq:holomorphic_volume_form}
\Omega_{m,n} \equiv \frac{  dx_0 \wedge \cdots \wedge dx_{m} \wedge d \log(y_1) \wedge \cdots \wedge d \log (y_n) }{d \left( \prod x_i - \sum y_j \right) ,}
\end{equation}
which by definition satisfies
\begin{equation}
d \left( \prod x_i - \sum y_j \right)\wedge\Omega_{m,n} = dx_0 \wedge \cdots \wedge dx_{m} \wedge d \log(y_1) \wedge \cdots \wedge d \log (y_n)
\end{equation}
in the normal bundle to $X_{m,n}$. For the next statement, we identify $H_1$ with the complement of the origin in $\bC \setminus \{1\}$, which we equip with the holomorphic volume form
\begin{equation}
  \Omega_1 \equiv \frac{dw}{w-1} = d\log(w-1).
\end{equation}
\begin{lem}
  The restriction of $\Omega_{m,n}$ to the complement of $D$ agrees with the wedge product of the forms $d\log x$ on the $\bC^*$ factors with the forms $\Omega_{1}$ on the $H_1$ factors.
\end{lem}
\begin{proof}
  Note first that
  \begin{equation}
  \begin{aligned}
  d \left( \prod x_i - \sum y_j \right) \wedge d\log x_1 \wedge\dotsb&\wedge d\log x_m \,\wedge\, \Omega_{1,1} \wedge\dotsb\wedge \Omega_{1,n}\\
  & =
  dx_0 \wedge\dotsb\wedge dx_m \wedge \Omega_{1,1} \wedge\dotsb\wedge \Omega_{1,n},
  \end{aligned}
  \end{equation}
  since the $\Omega_{1,j}$ terms depend only on the $y$ coordinates and there are already $n$ of them. Now,
  \begin{equation*}
  \begin{aligned}
  \Omega_{1,1} &\wedge\dotsb\wedge \Omega_{1,n} \\
    &=\left(d\log \frac{y_1}{y_2}\right) \wedge\dotsb\wedge \left(d\log \frac{y_1+\dotsb+y_{n-1}}{y_n}\right) \wedge \left(d\log(y_1+\dotsb+y_n)\right)\\
    &=\left(\frac{dy_1}{y_1} - \frac{dy_2}{y_2}\right) \wedge\dotsb\wedge \left(\frac{dy_1+\dotsb+dy_{n-1}}{y_1+\dotsb+y_{n-1}} - \frac{dy_n}{y_n}\right) \wedge \frac{dy_1+\dotsb+dy_n}{y_1+\dotsb+dy_n}\\
    &=
    \begin{vmatrix}
      \frac{1}{y_1} & \frac{1}{y_1+y_2} & \cdots & \frac{1}{y_1+\dotsb+y_n} \\
      -\frac{1}{y_2} & \frac{1}{y_1+y_2} & \cdots & \frac{1}{y_1+\dotsb+y_n} \\
      0 & -\frac{1}{y_3} & \cdots & \frac{1}{y_1+\dotsb+y_n} \\
      \vdots & \vdots & \ddots & \vdots \\
      0 & 0 & \cdots & \frac{1}{y_1+\dotsb+y_n}
     \end{vmatrix}
      dy_1 \wedge\dotsb\wedge dy_n,
  \end{aligned}
  \end{equation*}
  and the determinant equals $\prod\frac{1}{y_j}$ by  Gaussian elimination.
\end{proof}

Since the Lagrangian $L_1 \subset H_1$ is graded with respect to $\Omega_1$, we immediately conclude:
\begin{cor}
  The Lagrangian $L_{m,n}$ is graded with respect to the holomorphic volume form $\Omega_{m,n}$. \qed
\end{cor}

\subsection{Bounding cochains}
\label{sec:bound-coch}

In order to associate to an immersed Lagrangian an object of the Fukaya category with coefficients in a commutative ring $\bfk$, one needs to choose an element of its deformation space, i.e. a local system as well as a linear combination of degree $1$ generators of the self-Floer cochains, with the property that the associated weighted count of holomorphic discs with one output vanishes. The weighted count is encoded by a function, with value the Floer cochains of degree $2$, defined on the deformation space. The simplest instance of this deformation space occurs when restrict attention to local systems of rank $1$ over $L$, in which case it consists of a product of copies of $\bfk$ indexed by the degree $1$ generators of the self-Floer cochains which are associated to self-intersections, with copies of $\bfk^*$ indexed by the generators of first homology; we refer to this function as the \emph{potential function} and its zero locus as the space of bounding cochains. A further natural condition to impose is that the resulting object be a non-trivial object of the Fukaya category, and we shall call the corresponding space the space of non-trivial bounding cochains.

We shall not perform an exhaustive analysis of the space of bounding cochains on $L_{m,n}$, as this will not be necessary for the purpose of understanding mirror symmetry. Instead, we shall focus our attention on a specific subspace which will turn out to be naturally identifiable with the mirror space.
\begin{rem}
  The usefulness of restricting to a subspace of the space of bounding cochains is already apparent in the case of the $1$-dimensional pair of pants $H_1$; a priori, the potential function is defined on a product $\bk_u \times \bk_v \times \bk^*_\mu$, and it is not too difficult to compute that the potential function is given by $u v (1 - \mu)$. The set of non-trivial objects is further constrained by imposing the equation $u v = 0$. The outcome is that the space of non-trivial bounding cochains is a product of $\bk^*$ with the union of the coordinate lines in $\bk_u \times \bk_v  $. In terms of the mirror space $\Pi_1$, this can be explained as follows: the moduli space parametrised by this Maurer-Cartan space is that of triples $(p_1, p_2, \mu)$, with $p_i \in Y$ and $\mu \cdot p_1 = p_2$. Thus, specialising to $\mu = 1$ recovers a copy of $\Pi_1$ embedded diagonally.

  Restricting this space will prove even more useful in higher dimension, because $L_1$ has self-intersections of negative degree, which causes $CF^1(L_{m,n})$ to be quite large.
\end{rem}

\subsubsection{A distinguished subspace of the deformation space}
\label{sec:dist-moduli-space}

Let us fix some notation: Seidel's Lagrangian $L_1$ has three self-intersections; we denote the corresponding pairs of generators of the self-Floer complex of $L_1$ by $(s,\partial_s)$, $(t, \partial_t)$, and $(c, \partial_c)$; our notational convention is that the elements $(s,t,c)$ have even degree, and the elements $(\partial_s, \partial_t, \partial_c)$ have odd degree.

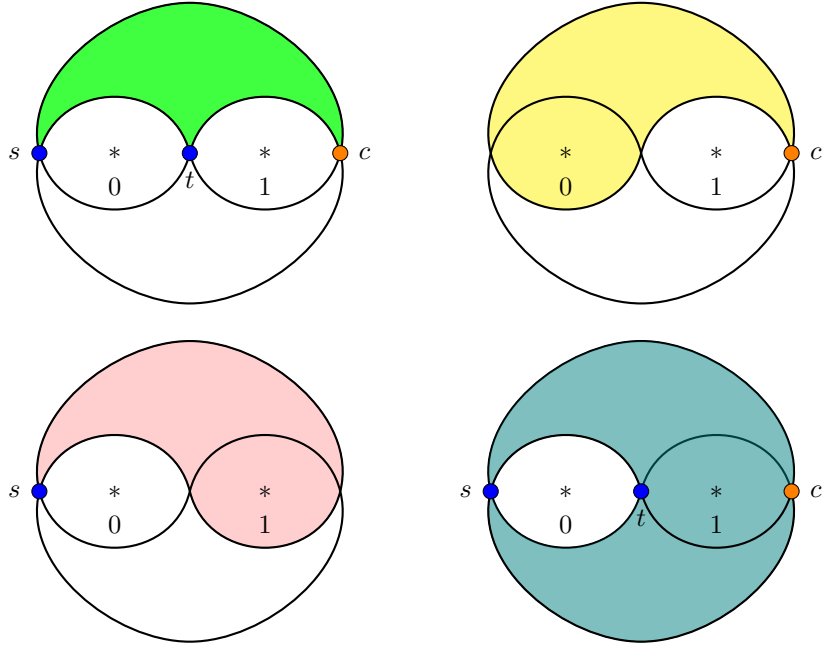
\begin{figure}
  \centering
  \begin{tikzpicture}

      \node[label = right:{$c$}] (c) at (3,0) {};
      \draw[thick, fill = yellow, fill opacity = 0.5] (3,0) .. controls (3.25,1) and (2,2) .. (1,2) .. controls (0,2) and (-1.25,1) .. (-1,0) .. controls (-.75,-1) and (.75, -1) ..  (1,0) .. controls (1.25,1)  and (2.75,1) .. (3,0);

  \begin{scope}[yscale = -1, xscale = 1]
    \draw[thick] (3,0) .. controls (3.25,1) and (2,2) .. (1,2) .. controls (0,2) and (-1.25,1) .. (-1,0) .. controls (-.75,-1) and (.75, -1) ..  (1,0) .. controls (1.25,1)  and (2.75,1) .. (3,0);
  \end{scope}
  \draw [fill=orange] (c) circle (0.1);
      \node[label = below:{$0$}] (0) at (0,0) {$\ast$};
      \node[label = below:{$1$}] (1) at (2,0) {$\ast$};
\begin{scope}[shift = {(-6,0)}]
\node[label = left:{$s$}] (s) at (-1,0) {};
      
      \node[label = below:{$t$}] (t) at (1,0) {};
      \node[label = right:{$c$}] (c) at (3,0) {};
      
      \draw[thin, fill = green, fill opacity = 0.75] (3,0) .. controls (3.25,1) and (2,2) .. (1,2) .. controls (0,2) and (-1.25,1) .. (-1,0) .. controls (-.75,1) and (.75, 1) ..  (1,0) .. controls (1.25,1)  and (2.75,1) .. (3,0);
      \draw[thick] (3,0) .. controls (3.25,1) and (2,2) .. (1,2) .. controls (0,2) and (-1.25,1) .. (-1,0) .. controls (-.75,-1) and (.75, -1) ..  (1,0) .. controls (1.25,1)  and (2.75,1) .. (3,0);
  \begin{scope}[yscale = -1, xscale = 1]
    \draw[thick] (3,0) .. controls (3.25,1) and (2,2) .. (1,2) .. controls (0,2) and (-1.25,1) .. (-1,0) .. controls (-.75,-1) and (.75, -1) ..  (1,0) .. controls (1.25,1)  and (2.75,1) .. (3,0);
  \end{scope}
  \draw [fill=blue] (s) circle (0.1);
  \draw [fill=blue] (t) circle (0.1);
  \draw [fill=orange] (c) circle (0.1);
      \node[label = below:{$0$}] (0) at (0,0) {$\ast$};
      \node[label = below:{$1$}] (1) at (2,0) {$\ast$};
    \end{scope}

    \begin{scope}[shift ={(0,-4.5)}]

      \draw[thick] (3,0) .. controls (3.25,1) and (2,2) .. (1,2) .. controls (0,2) and (-1.25,1) .. (-1,0) .. controls (-.75,-1) and (.75, -1) ..  (1,0) .. controls (1.25,1)  and (2.75,1) .. (3,0);
      \draw [thin, fill = teal, fill opacity = 0.5]      (3,0) .. controls (3.25,1) and (2,2) .. (1,2) .. controls (0,2) and (-1.25,1) .. (-1,0) .. controls (-.75,1) and (.75, 1) ..  (1,0)  .. controls (.75,-1) and (-.75, -1) .. (-1,0) .. controls (-1.25,-1)  and (0,-2) .. (1,-2) .. controls (2,-2) and (3.25,-1) .. (3,0) ;

  \begin{scope}[yscale = -1, xscale = 1]
    \draw[thick] (3,0) .. controls (3.25,1) and (2,2) .. (1,2) .. controls (0,2) and (-1.25,1) .. (-1,0) .. controls (-.75,-1) and (.75, -1) ..  (1,0) .. controls (1.25,1)  and (2.75,1) .. (3,0);
  \end{scope}
        \node[label = left:{$s$}] (s) at (-1,0) {};
      
      \node[label = below:{$t$}] (t) at (1,0) {};
      \node[label = right:{$c$}] (c) at (3,0) {};

  \draw [fill=blue] (s) circle (0.1);
  \draw [fill=blue] (t) circle (0.1);
  \draw [fill=orange] (c) circle (0.1);
      \node[label = below:{$0$}] (0) at (0,0) {$\ast$};
      \node[label = below:{$1$}] (1) at (2,0) {$\ast$};
\begin{scope}[shift = {(-6,0)}]
\node[label = left:{$s$}] (s) at (-1,0) {};
      
      \draw[thin, fill = pink,  fill opacity = 0.75] (3,0) .. controls (3.25,1) and (2,2) .. (1,2) .. controls (0,2) and (-1.25,1) .. (-1,0) .. controls (-.75,1) and (.75, 1) ..  (1,0) .. controls (1.25,-1)  and (2.75,-1) .. (3,0);
      
      \draw[thick] (3,0) .. controls (3.25,1) and (2,2) .. (1,2) .. controls (0,2) and (-1.25,1) .. (-1,0) .. controls (-.75,-1) and (.75, -1) ..  (1,0) .. controls (1.25,1)  and (2.75,1) .. (3,0);
  \begin{scope}[yscale = -1, xscale = 1]
    \draw[thick] (3,0) .. controls (3.25,1) and (2,2) .. (1,2) .. controls (0,2) and (-1.25,1) .. (-1,0) .. controls (-.75,-1) and (.75, -1) ..  (1,0) .. controls (1.25,1)  and (2.75,1) .. (3,0);
  \end{scope}
  \draw [fill=blue] (s) circle (0.1);
      \node[label = below:{$0$}] (0) at (0,0) {$\ast$};
      \node[label = below:{$1$}] (1) at (2,0) {$\ast$};
    \end{scope}    
  
    \end{scope}    

  \end{tikzpicture}
  \caption{The four holomorphic discs with allowed inputs and one output (the reflection of each disc across the horizontal axis may also contribute).}
  \label{fig:Seidel-Lagrangian-1-and-3-gon}
\end{figure}

The easiest way to determine the degrees of these generators is as follows: observe that our holomorphic volume form on $H_1$ extends meromorphically to $\bP^1$, with simple poles at $1$ and $\infty$. Since the generator $c$ is the output of a pair of genuine teardrops passing through $0$, one of which is shown on the upper right of Figure \ref{fig:Seidel-Lagrangian-1-and-3-gon}, and the holomorphic volume form extends without singularities to the origin, this generator has degree $2$ (the other teardrop is obtained by vertical reflection). The generator $s$ bounds a pair of teardrops passing through the puncture at $1$, so it has degree $0$. Similarly, $t$ bounds a pair of teardrops surrounding the puncture at infinity, so it also has degree $0$.  The degree of the Poincar\'e dual generators are therefore determined to be $1$, $1$, and $-1$.

Floer generators on the product $L_1^n$ are given as tensor products of generators on each component. Write $s_j$ for the generator
\begin{equation}
\underbrace{e\otimes\dotsm\otimes e}_{j-1} \otimes s \otimes \underbrace{e\otimes\dotsm\otimes e}_{n-j},    
\end{equation}
where $e\in CF^0(L_1)$ is the Morse unit. Similarly, write $\partial_{s_j}$, $t_j$, $\partial_{t_j}$, $c_j$, and $\partial_{ c_j}$ for the other simple self-intersections.

The torus $T^m\subset(\bC^*)^m$ has a natural splitting. We call the corresponding monodromy generators for rank-$1$ local systems $\{\partial_{\zeta_i}\}_{i=1}^{m}$. We view this as a $\bk^*$-coordinate on the space of branes supported on $L_{m,n}$.

We now distinguish a subspace $\mathrm{Def}_{m,n}$  of the deformation space of $L_{m,n}$ on which we will show that the potential function vanishes. This space is given by
\begin{equation}
\mathrm{Def}_{m,n} \equiv \{ \sum_{i=1}^m\rho_i\partial_{\zeta_i} + \sum_{	j=1}^n\left(  u_j\partial_{s_j} + v_j\partial_{t_j} \right) | \rho_i \in \bk^* \textrm{ and } u_j, v_j \in \bk \}.
\end{equation}

\subsubsection{Discs with boundary on $L_{m,n}$}
\label{sec:disc-bdry-Lmn}

Our plan to compute the Maurer--Cartan formula is to reduce the count of discs in $X_{m,n}$ to disc counts in $\bC$. To that end, we use the following result, which uses the projection map which appears in Lemma \ref{lem:product_decompositions}:
\begin{lem}
	A holomorphic map $u\colon D^2\to X_{m,n}$ with boundary on $L_{m,n}$ is determined by its projection
	\begin{equation}
	(x_1,\dotsc,x_m,q_1,\dotsc,q_n)\circ u
	\end{equation}
	to $\bC^{m} \times \bC^{n} $, which is a holomorphic disc with boundary on the product $T^m \times L_{1}^n$.

\end{lem}
\begin{proof}
	Note that $L_{m,n}$ lives away from $D$, so a holomorphic polygon $u$ with boundary on $L_{m,n}$ cannot live entirely in $D$. By unique continuation, $u$ intersects $D$ at only finitely many points, which implies that $u$ is determined on a dense open subset by the projection.

\end{proof}

Since we are considering a product Lagrangian in  $\bC^{m} \times \bC^{n}$, all holomorphic discs split.  The $\bC^m$-component is a disc with boundary on the Clifford torus, which is well-understood. Let us analyze the possible discs in the remaining $\bC$-components. These are discs with boundary on $L_1$, and the assumption that they contribute to the Maurer--Cartan function on $\mathrm{Def}_{m,n}$ means that all inputs are either $e$, $s$, or $t$. The next statement is easiest to understand while referring to Figure  \ref{fig:Seidel-Lagrangian-1-and-3-gon}:

\begin{lemma}\label{lem:mu0 disc classification in 1d}
	Any nonconstant holomorphic disc $u\colon D^2\to\bC$ with boundary on $L_1$ with zero or one output and with inputs from $\{s,t\}$ corresponds to one of the following possibilities:
	\begin{description}
		\item[Green]\label{eq:mu0 3-gon} A 3-gon in the complement of $0$ and $1$.
		\item[Yellow]\label{eq:mu0 1-gon through 0} A 1-gon intersecting $0$ with output on $c$.
		\item[Pink]\label{eq:mu0 1-gon through 1} A \emph{branched cover} of a 1-gon intersecting $1$ with one corner at $s$.
		\item[Blue]\label{eq:mu0 4-gon} A 4-gon intersecting $1$ with two corners at $s$, one corner at $t$, and one corner at $c$.
	\end{description}
	\qed
\end{lemma}

We are now prepared to classify the possible discs contributing to the Maurer--cartan equation. We note that the following result in particular asserts that the Blue disc appearing
above does not contribute to the desired count:
\begin{prop}\label{prop:m0-disc-classification}
	Any holomorphic disc $u$ contributing to the Maurer--Cartan function $\mu^0$ on $\mathrm{Def}_{m,n}$ is nonconstant in at most two of the factors
	\begin{equation}
	\bC_{q_1}, \dotsc, \bC_{q_n}, \bC^m
	\end{equation}
	and those factors must be consecutive in the above ordering. These discs come in one of three types.
	\begin{enumerate}
		\item \label{case:m0-disc-0-intersection}$u$ does not intersect $D$. It has one nonconstant factor, and that factor is a Green disc from Lemma \ref{lem:mu0 disc classification in 1d}.
    \item\label{case:m0-rho-disc} $u$ intersects $D$, and it has two nonconstant $q$ factors. The first of those factors is a Yellow disc in $\bC_{q_j}$, and the second is a \emph{simple} Pink disc in $\bC_{q_{j+1}}$.
    \item\label{case:m0-hybrid-disc} $u$ intersects $D$, and its first nonconstant factor is a Yellow disc in $\bC_{q_n}$. Its $\bC^m$ component is either constant or a Maslov $2$ disc with boundary on the Clifford torus.
	\end{enumerate}
\end{prop}
\begin{proof}
  A disc which contributes to $\mu^0$ must have a degree $2$ generator as its output. Lemma \ref{lem:mu0 disc classification in 1d} implies that this output must be the product of the generator $c$ in one factor with the identity in the others. We can then consider the cases according to the intersection number with $D$.

  First, if the intersection number vanishes, then the only possibility is a Green disc in one factor with all other factors constant. Next, if the intersection number is $1$, and the component meeting $u$ is $D_j$, then  Lemma \ref{lem:iterative_description_divisors} implies that $u$ must agree with the Yellow disc in the factor labelled by $q_{j-1}$. Since the Yellow disc has an output labelled $c$, the output of any other factor must have degree $0$, i.e. be a marked point along the smooth part. There are now three situations to consider:
  \begin{enumerate}
  \item if $k < n$, then by the definition of $D_j$, the projection of $u$ to the $q_j$ factor must be non-trivial, and pass through the point $1$. Given the above constraint that its output must have degree $0$, this second factor must be a Pink disc.
  \item if $j=n$, and $u$ meets the component of the normal crossing divisor corresponding to $x_0=0$, then the projection to $\bC^m$ must be trivial.
  \item if $j=n$ and $u$ meets the component of the normal crossing divisor corresponding to $x_i=0$ for $i \neq 0$, then the projection to $\bC^m$ is a disc of index $2$ in the factor labelled by $i$.
  \end{enumerate}

  We preclude the possibility of a contribution from discs whose intersection number with $D$ is greater than $1$ as follows: on the one hand, there cannot be any factor in which there is a multiply-covered disc, because this would imply a contribution of a disc with input the Poincar\'e dual class to $c$, which contradicts our constraints about the inputs. On the other hand, if $u$ intersects more than one component of $D$, then there is more than one factor in which the output is labelled by $c$, contradicting the fact that the output must have total degree $2$.
      \end{proof}

Clearly, in the first case above, the disc exists and is uniquely defined. In the second case, we have a map
\begin{equation}
	u\colon D^2 \to \bC \times T,
\end{equation}
where $T$ is either $\bC$ or $\bC^m$. We wish to prove existence and uniqueness in this case as well.

\begin{lemma} \label{lem:regularity_discs}
	Let $\Sigma$ be a holomorphic disc with at least one boundary marked points and one interior marked point. For any collection of discs
	\begin{equation}
		u_j \colon \Sigma \to \bC_{q_j}, \qquad v\colon \Sigma \to \bC^m
	\end{equation}
	as in the second or third case of Proposition \ref{prop:m0-disc-classification}, mapping the interior marked point to the divisor $D$, there is a unique lift
	\begin{equation}
		u \colon \Sigma \to X_{m,n}.
              \end{equation}
Moreover,  the corresponding holomorphic disc with boundary on $L_{m,n}$ is regular.
\end{lemma}
\begin{proof}
  The interior marked point $z$ is the only point which could possibly map to $D$, so we know that there is a uniquely defined map $u_0 \colon \Sigma \setminus z \to X_{m,n} \setminus D$. By the removable singularities theorem, it suffices to show that the image of $u_0$ is bounded away from the ends of $\bC^{m+1} \times \left( \bC^* \right)^n$. Concretely, this amounts to checking that the functions $x_i \circ u_0$, $y_j \circ u_0$, and $\frac 1 {y_j \circ u_0}$ are each bounded. To that end, note that $x_1 \circ u_0$ through $x_m \circ u_0$ are bounded by assumption, while the others we can bound by formula. Indeed,
  \begin{align*}
    x_0 &= \frac { q_n } { x_1 \dotsm x_m } \\
    y_j &= \begin{cases}
      \frac { q_n - 1 } { q_{n-1} } &\text{ if }j = n \\
      \frac { q_j - 1 } { q_{j-1} } y_{j+1} &\text{ if } 1 < j < n \\
      \left( q_1 - 1 \right) y_2 &\text{ if } j = 1.
    \end{cases}
  \end{align*}

  Let us examine each of these formulas in turn. The formula for $x_0$ gives
  \begin{equation}
  x_0 \circ u_0 = \frac{u_n}{\prod v_i}.
  \end{equation}
  The denominator is either a nonzero constant or we are in case \ref{case:m0-hybrid-disc} of Proposition \ref{prop:m0-disc-classification}, in which case at most one $v_i$ is allowed a single simple zero by the Maslov 2 assumption. If this happens, then that zero is an intersection point with $D$, which by assumption means that it occurs at the marked point and that $u_n$ also has a simple zero at the marked point.  The situation for $y_j$ in case \ref{case:m0-rho-disc} is similar, using the fact that $u_j-1$ is allowed has simple $0$ precisely when $u_{j-1}$ has a simple $0$.

To prove regularity, observe that since the inputs have Maslov index $1$ and the output has Maslov index $2$, the holomorphic disc under consideration has Fredholm index equal to the rank of the space of holomorphic vector fields on the domain, which is $1$ in case \ref{case:m0-rho-disc} (because there is one input and one output), and $2$ in case \ref{case:m0-hybrid-disc} (because there is no input). It thus suffices to show that the kernel of the linearisation of the $\dbar$ operator (which are holomorphic vector field on the domain valued in the tangent space of the fibre) consists precisely of such vector fields in the domain.  This follows from the same analysis that yields the above description of the space of holomorphic discs: such vector fields split as a product away from the marked point, but the matching condition along the marked point shows that the vector fields must move the marked points in the same way, so they either correspond to an automorphism of the domain, or to an infinitesimal deformation. No such deformations exist because the curves in each factor are unique.
\end{proof}

\subsubsection{Proof that the curvature vanishes}
\label{sec:comp-maur-cart}

We conclude this section with the following computation:
\begin{lem} \label{lem:curvature_vanishes}
Every element of $\mathrm{Def}(L_{m,n})$ is a bounding cochain.
\end{lem}
\begin{proof}
  We consider the contributions of the three cases in Proposition \ref{prop:m0-disc-classification}, which we will see vanish by symmetry.

\subsubsection*{Case \ref{case:m0-disc-0-intersection}}
The only holomorphic polygons in any component which don't intersect $D$ are constants and the 3-gons of the form appearing in Figure \ref{fig:Seidel-Lagrangian-1-and-3-gon} and their reflections. The contribution of these curves to $\mu_0$ is given by:
\begin{equation}
\sum_{j=1}^n \left(u_jv_j\cdot c_j - u_jv_j\cdot c_j \right)= 0.
\end{equation}
The sign is determined by the computation of the product in \cite{Seidel2011}: given a choice of orientation of the Lagrangian, together with a distinguished marked point, the sign is given as in \cite{Seidel2008,Abouzaid2008} by a combinatorial count depending on the difference between the orientation of the Lagrangian and the boundary orientation of the disc along each segment except the one which ends at the output with respect to the counterlockwise orientation, together with a factor of $\pm 1$ according to whether the marked point lies on the boundary of the disc. In our case, the boundary of $L_1$ is partitioned between the two possible Green discs, so only one of them can contain the marked point, which implies that the sign contributions are opposite.

\subsubsection*{Case \ref{case:m0-rho-disc}}

These discs have input $s_{j+1}$, and output $c_{j}$. A choice of point in the factor of $L_{m,n}$ labelled by $q_{j+1}$, representing the unit in the cohomology of $S^1$ distinguishes one of the two Pink discs as contributing to the count (since the union of the boundaries of these two discs together bijectively cover the complement of the self-intersections). This means that the contributions in this case are given by
\begin{equation}
\sum_{j=1}^n \left(u_{j+1} \cdot c_j - u_{j+1}\cdot c_j \right)= 0,
\end{equation}
with the two terms coming from the two possible Yellow discs in the factor labelled by $q_{j}$. The sign is determined as above by the contribution in each factor.

\subsubsection*{Case \ref{case:m0-hybrid-disc}}
In this case, there are again two discs corresponding to the two possible Yellow discs in the factor labelled by $q_j$, and their contribution is
\begin{equation}
   \left( 1 + \sum_{i=1}^{m} \rho_i\right) \cdot c_{n} - \left( 1 + \sum_{i=1}^{m} \rho_i\right) \cdot c_{n} = 0.
\end{equation}
In the above expression, the terms which equal $ \pm 1 \cdot c_n$ arise from the two discs which meet the $x_0=0$ component of the normal crossing divisor $\prod x_i =0$, and the terms $\pm \rho_i \cdot c_n$ from the discs with meet the component $x_i =0$ for $i \neq 0$. 

\end{proof}

\subsection{Floer homology with the $0$-section}
\label{sec:floer-homology-with}

The outcome of Lemma \ref{lem:curvature_vanishes} is that every element of $\mathrm{Def}_{m,n}$ defines an object of the Fukaya category. We shall eventually identify the quasi-isomorphism class of these objects; in this section, we shall, as a first step, show that they do not vanish by computing the Floer cohomology with a non-compact Lagrangian which is mirror to the structure sheaf:
\begin{defin}
  The \emph{zero section} of $X_{m,n}$ is the real Lagrangian $L^+_{m,n}$ given by the conditions
  \begin{equation}
    x_i, y_j \in (0,\infty).    
  \end{equation}
\end{defin}
To simplify the notation, we shall usually write $L^+$ for this Lagrangian, omitting the subscripts. In order for $L^+$ to define an object of the Fukaya category, we must show that it is well-behaved at infinity:
\begin{lem}
  The inclusion of $L^+$ in $X_{m,n}$ is proper.
\end{lem}
\begin{proof}
  Any limit point of $L^+$ in $X_{m,n}$ must have some coordinate $x_i$ that vanishes. By the defining equation of $X_{m,n}$, this implies that $1 + \sum y_j =0$, which means that some coordinate $y_j$ must be negative, contradicting the assumption that $0< y_j$.
\end{proof}
One also needs to ensure that $L^+$ is invariant under the Liouville flow of the chosen primitive for the symplectic form; this follows immediately from choosing the primitive to be real.

We note that $L^+$ is contained in the complement of the divisor $D$, i.e. in the open subset $(\bC^{*})^{m} \times H^n_1 $ of $X_{m,n}$, and agrees with the product of the lines  $0 < x_i$  and $1 < q_j$ in these coordinates. In drawing the figures, we shall describe $L^+$ as a product of vertical arcs emanating from the marked point labelled $1$; this can be achieved by deforming $L^+$ by Hamiltonian isotopy (deforming the primitive as well to preserve the fact that $L^+$ is both a product of Lagrangians and is conical).

The factors of the Lagrangians $L^+$ and $L_{m,n}$ corresponding to the coordinates $q_i$ intersect transversely at two points which we label $a_i$ and $b_i$, while the factors in $(\bC^*)^m$ meet at one point. For an appropriate choice of grading, the degree of $a_i$ is $0$ and that of $b_i$ is $1$.

\begin{figure}
  \centering
  \begin{tikzpicture}
            \node[label = {[shift={(0.3,-.2)}]$a$}] (a) at (2,.75) {};
      \node[label = right:{$b$}] (b) at (2,1.75) {};

      \draw[thick] (3,0) .. controls (3.25,1) and (2,2) .. (1,2) .. controls (0,2) and (-1.25,1) .. (-1,0) .. controls (-.75,-1) and (.75, -1) ..  (1,0) .. controls (1.25,1)  and (2.75,1) .. (3,0);
      
      \begin{scope}
        \clip (2,-2) rectangle (-1.5,2);
      \draw[thick, fill = yellow, fill opacity = 0.5] (3,0) .. controls (3.25,1) and (2,2) .. (1,2) .. controls (0,2) and (-1.25,1) .. (-1,0) .. controls (-.75,-1) and (.75, -1) ..  (1,0) .. controls (1.25,1)  and (2.75,1) .. (3,0);  
      \end{scope}
     
  \begin{scope}[yscale = -1, xscale = 1]
    \draw[thick] (3,0) .. controls (3.25,1) and (2,2) .. (1,2) .. controls (0,2) and (-1.25,1) .. (-1,0) .. controls (-.75,-1) and (.75, -1) ..  (1,0) .. controls (1.25,1)  and (2.75,1) .. (3,0);
  \end{scope}
    \draw[thick, gray] (2,0)  -- (2,2);
   \draw [fill=blue] (a) circle (0.1);
  \draw [fill=orange] (b) circle (0.1);
      \node[label = below:{$0$}] (0) at (0,0) {$\ast$};
      \node[label = below:{$1$}] (1) at (2,0) {$\ast$};
    
      \begin{scope}[shift = {(-6,0)}]
            \node[label = {[shift={(0.3,-.2)}]$a$}] (a) at (2,.75) {};
      \node[label = right:{$b$}] (b) at (2,1.75) {};

  \node[label = left:{$s$}] (s) at (-1,0) {};
      
      \node[label = below:{$t$}] (t) at (1,0) {};

      \draw[thin] (3,0) .. controls (3.25,1) and (2,2) .. (1,2) .. controls (0,2) and (-1.25,1) .. (-1,0) .. controls (-.75,1) and (.75, 1) ..  (1,0) .. controls (1.25,1)  and (2.75,1) .. (3,0);
      \begin{scope}
        \clip (2,-2) rectangle (-1.5,2);
      \draw[thin, fill = green, fill opacity = 0.75] (3,0) .. controls (3.25,1) and (2,2) .. (1,2) .. controls (0,2) and (-1.25,1) .. (-1,0) .. controls (-.75,1) and (.75, 1) ..  (1,0) .. controls (1.25,1)  and (2.75,1) .. (3,0);  
      \end{scope}
      
      \draw[thick] (3,0) .. controls (3.25,1) and (2,2) .. (1,2) .. controls (0,2) and (-1.25,1) .. (-1,0) .. controls (-.75,-1) and (.75, -1) ..  (1,0) .. controls (1.25,1)  and (2.75,1) .. (3,0);
  \begin{scope}[yscale = -1, xscale = 1]
    \draw[thick] (3,0) .. controls (3.25,1) and (2,2) .. (1,2) .. controls (0,2) and (-1.25,1) .. (-1,0) .. controls (-.75,-1) and (.75, -1) ..  (1,0) .. controls (1.25,1)  and (2.75,1) .. (3,0);
  \end{scope}
  \draw [fill=blue] (s) circle (0.1);
  \draw [fill=blue] (t) circle (0.1);
      \node[label = below:{$0$}] (0) at (0,0) {$\ast$};
      \node[label = below:{$1$}] (1) at (2,0) {$\ast$};
      \draw[thick, gray] (2,0)  -- (2,2);
        \draw [fill=blue] (a) circle (0.1);
  \draw [fill=orange] (b) circle (0.1);
    \end{scope}

\begin{scope}[shift = {(-6,-4.5)}]

  \node[label = left:{$s$}] (s) at (-1,0) {};   
  \begin{scope}[yscale=-1]
    \draw[thin, fill = pink,  fill opacity = 0.75] (3,0) .. controls (3.25,1) and (2,2) .. (1,2) .. controls (0,2) and (-1.25,1) .. (-1,0) .. controls (-.75,1) and (.75, 1) ..  (1,0) .. controls (1.25,-1)  and (2.75,-1) .. (3,0);
    
  \end{scope}
      
      \draw[thick] (3,0) .. controls (3.25,1) and (2,2) .. (1,2) .. controls (0,2) and (-1.25,1) .. (-1,0) .. controls (-.75,-1) and (.75, -1) ..  (1,0) .. controls (1.25,1)  and (2.75,1) .. (3,0);
  \begin{scope}[yscale = -1, xscale = 1]
    \draw[thick] (3,0) .. controls (3.25,1) and (2,2) .. (1,2) .. controls (0,2) and (-1.25,1) .. (-1,0) .. controls (-.75,-1) and (.75, -1) ..  (1,0) .. controls (1.25,1)  and (2.75,1) .. (3,0);
  \end{scope}
  \draw [fill=blue] (s) circle (0.1);
      \node[label = below:{$0$}] (0) at (0,0) {$\ast$};
      \node[label = below:{$1$}] (1) at (2,0) {$\ast$};
           \draw[thick, gray] (2,0)  -- (2,2);
  \node[label = {[shift={(0.3,-.2)}]$a$}] (a) at (2,.75) {};
      \draw [fill=black] (a) circle (0.1);

    \end{scope}   
\begin{scope}[shift = {(0,-4.5)}]

  \node[label = left:{$s$}] (s) at (-1,0) {};
            
  \begin{scope}
    \draw[thin, fill = pink,  fill opacity = 0.75] (3,0) .. controls (3.25,1) and (2,2) .. (1,2) .. controls (0,2) and (-1.25,1) .. (-1,0) .. controls (-.75,1) and (.75, 1) ..  (1,0) .. controls (1.25,-1)  and (2.75,-1) .. (3,0);
    
  \end{scope}
      
      \draw[thick] (3,0) .. controls (3.25,1) and (2,2) .. (1,2) .. controls (0,2) and (-1.25,1) .. (-1,0) .. controls (-.75,-1) and (.75, -1) ..  (1,0) .. controls (1.25,1)  and (2.75,1) .. (3,0);
  \begin{scope}[yscale = -1, xscale = 1]
    \draw[thick] (3,0) .. controls (3.25,1) and (2,2) .. (1,2) .. controls (0,2) and (-1.25,1) .. (-1,0) .. controls (-.75,-1) and (.75, -1) ..  (1,0) .. controls (1.25,1)  and (2.75,1) .. (3,0);
  \end{scope}
  \draw [fill=blue] (s) circle (0.1);
      \node[label = below:{$0$}] (0) at (0,0) {$\ast$};
      \node[label = below:{$1$}] (1) at (2,0) {$\ast$};
           \draw[thick, gray] (2,0)  -- (2,2);
      \node[label = right:{$b$}] (b) at (2,1.75) {};

  \draw [fill=black] (b) circle (0.1);

    \end{scope}   

    \begin{scope}[shift ={(-3,-9)}]
      \draw[thick] (3,0) .. controls (3.25,1) and (2,2) .. (1,2) .. controls (0,2) and (-1.25,1) .. (-1,0) .. controls (-.75,-1) and (.75, -1) ..  (1,0) .. controls (1.25,1)  and (2.75,1) .. (3,0);

      \begin{scope}
           \clip  (-1.5,2) -- (2,2) -- (2,.75) .. controls (1.25,.75) and (2.75,1) .. (3,0) -- (3.5,-2) -- (-1.5,-2) -- cycle;
\draw [thin, fill = teal, fill opacity = 0.5]      (3,0) .. controls (3.25,1) and (2,2) .. (1,2) .. controls (0,2) and (-1.25,1) .. (-1,0) .. controls (-.75,1) and (.75, 1) ..  (1,0)  .. controls (.75,-1) and (-.75, -1) .. (-1,0) .. controls (-1.25,-1)  and (0,-2) .. (1,-2) .. controls (2,-2) and (3.25,-1) .. (3,0) ;
           
      \end{scope}
      
  \begin{scope}[yscale = -1, xscale = 1]
    \draw[thick] (3,0) .. controls (3.25,1) and (2,2) .. (1,2) .. controls (0,2) and (-1.25,1) .. (-1,0) .. controls (-.75,-1) and (.75, -1) ..  (1,0) .. controls (1.25,1)  and (2.75,1) .. (3,0);
  \end{scope}
        \node[label = left:{$s$}] (s) at (-1,0) {};
      
      \node[label = below:{$t$}] (t) at (1,0) {};

  \draw [fill=blue] (s) circle (0.1);
  \draw [fill=blue] (t) circle (0.1);
      \node[label = below:{$0$}] (0) at (0,0) {$\ast$};
      \node[label = below:{$1$}] (1) at (2,0) {$\ast$};
                 \draw[thick, gray] (2,0)  -- (2,2);
   \node[label = {[shift={(0.3,-.2)}]$a$}] (a) at (2,.75) {};
      \node[label = right:{$b$}] (b) at (2,1.75) {};

      \draw [fill=blue] (a) circle (0.1);
  \draw [fill=orange] (b) circle (0.1);

    \end{scope}    

  \end{tikzpicture}
  \caption{The holomorphic discs that can a priori arise in each factor when computing in the differential of the Floer complex of the pair $L_{m,n}$ and $L^+_{m,n}$.}
  \label{fig:Seidel-Lagrangian-2-and-4-gon}
\end{figure}
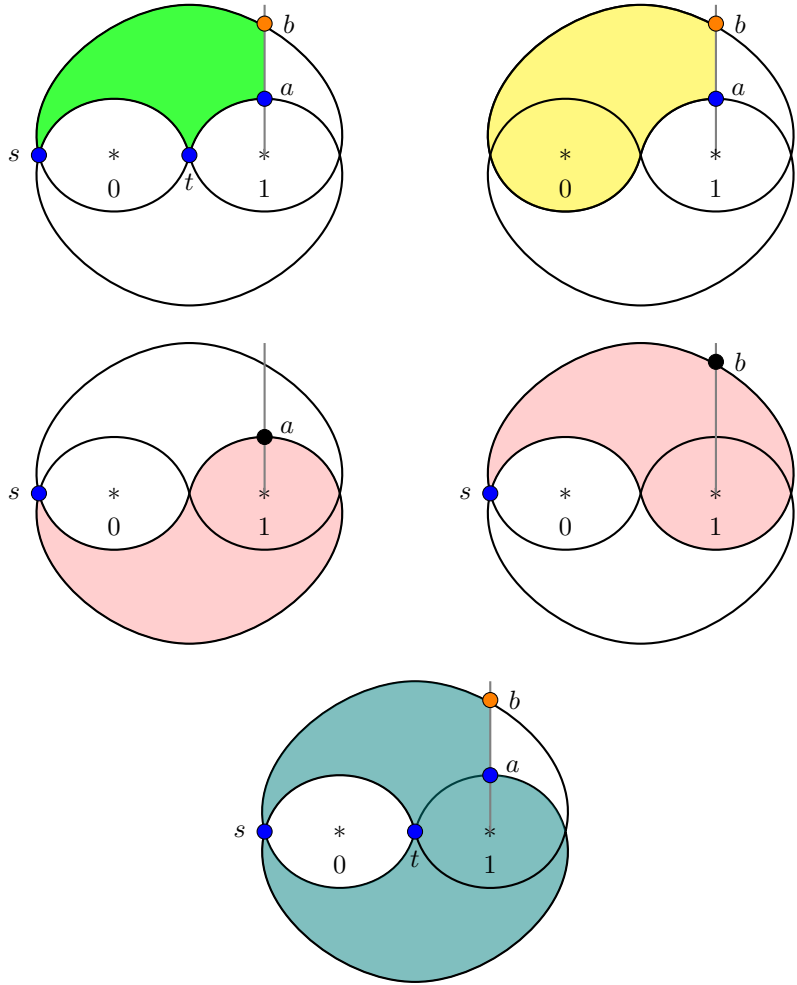

We now state the analogue of Lemma \ref{lem:mu0 disc classification in 1d}, where we now refer to Figure \ref{fig:Seidel-Lagrangian-2-and-4-gon}:
\begin{lem}
  \label{lem:mu1-disc-classification-1d}
  	Any nonconstant holomorphic map $u\colon \bR \times [0,1] \to\bC$ mapping $\bR \times \{0\}$ to $L_1$ and $\bR \times \{1\}$ to $L^+_{m,n}$, and with inputs along $L_1$ labelled by the self-intersections $\{s,t\}$ of $L_1$ has image given by one of the following four possibilities:
 	\begin{description}
		\item[Green]\label{eq:mu1 4-gon} A 4-gon in the complement of $0$ and $1$, with inputs $\{a,t,s\}$ and output $b$.
		\item[Yellow]\label{eq:mu1 2-gon through 0} A bi-gon intersecting $0$ with input $a$ and output $b$.
		\item[Pink]\label{eq:mu1 1-gon through 1} A 1-gon intersecting $1$ with (i) input  $\{a,s\}$ and output $a$ or (ii) input $\{b,s\}$ and output $\{b\}$.
		\item[Blue]\label{eq:mu1 5-gon} A 5-gon intersecting $1$ with inputs $\{a,s,t,s\}$ and output $\{b\}$.
	\end{description}
	\qed
\end{lem}

We then have the analogue of Proposition \ref{prop:m0-disc-classification}
\begin{prop}\label{prop:m1-disc-classification}
	Any holomorphic disc $u$ contributing to the differential in the Floer complex of $L_{m,n}$ equipped with a bounding cochain in $\Def_{m,n}$ and $L^+_{m,n}$ is nonconstant in at most two of the factors
	\begin{equation}
	\bC_{q_1}, \dotsc, \bC_{q_n}, \bC^m
	\end{equation}
	and those factors must be consecutive in the above ordering. These discs come in one of three types.
	\begin{enumerate}
		\item \label{case:m1-disc-0-intersection}$u$ does not intersect $D$. It has one nonconstant factor, and that factor is a Green disc from Lemma \ref{lem:mu1-disc-classification-1d}.
    \item\label{case:m1-rho-disc} $u$ intersects $D$, and it has two nonconstant $q$ factors. The first of those factors is a Yellow disc in $\bC_{q_j}$, and the second is a Pink curve in $\bC_{q_{j+1}}$.
    \item\label{case:m1-hybrid-disc} $u$ intersects $D$, and its first nonconstant factor is a Yellow disc in $\bC_{q_n}$. Its $\bC^m$ component is either constant or has Maslov index $2$, in which case it is non-trivial in exactly one component.
	\end{enumerate}
\end{prop}
\begin{proof}
  The argument is entirely analogous to the one for Proposition \ref{prop:m0-disc-classification}; we shall only remind the reader that the Blue discs do not contribute because there must be an accompanying Yellow disc in the previous factor, which implies that the difference between the degrees of the input and the output generators of the Floer cochains of  $L_{m,n}$ and $L^+_{m,n}$  is $2$, so that this homotopy class does not contribute to the differential.
\end{proof}

There is an existence and regularity statement for these discs which is exactly analogous to Lemma \ref{lem:regularity_discs}, which gives the following result, corresponding to Lemma \ref{lem:curvature_vanishes}: 
\begin{lem} \label{lem:computation-differential} 
  The differential in the Floer cochains of  $L_{m,n}$ and $L^+_{m,n}$  associated to an element of $\mathrm{Def}_{m,n}$ vanishes if and only if
  \begin{align} \label{eq:equation-moduli-1}
    u_j v_j - u_{j+1} & = 0 \textrm{ for } 0 \leq j < n \\ \label{eq:equation-moduli-2}
        u_1 \cdot \prod_{j=1}^{n} v_j & = 1 +  \sum_{i=1}^{m} \rho_i.    
  \end{align}
\end{lem}
\begin{proof}
  We compute that the differential is given by
   \begin{equation}
    a_j \mapsto
    \begin{cases}
      \left(u_j v_j - u_{j+1} \right) \cdot b_j    & \textrm{ if }  j < n \\
      \left( u_{n} v_{n} - \left( 1 +  \sum_{i=1}^{m} \rho_i \right) \right) \cdot b_n & \textrm{ if } j=n,
    \end{cases}
  \end{equation}
  where the specific sign is achieved by placing the marked point on $L_1$, which records the fact that we are using the bounding $\Spin$ structure, along the upper segment connecting $s$ and $t$. 
\end{proof}

Define $\cM_{m,n} \subset  \mathrm{Def}_{m,n}$ to consist of those bounding cochains on which the differential in the Floer complex between $L^+_{m,n}$ and $L_{m,n}$ identically vanishes.
\begin{rem}
  In fact, one can prove from the above description of the differential and from the fact that $ L^+_{m,n}$ corresponds under mirror symmetry to the structure sheaf, which we prove in Section \ref{sec:homol-mirr-symm} below, that $\cM_{m,n} $ corresponds to the set of elements of $\mathrm{Def}_{m,n} $ whose associated object of the Fukaya category is non-trivial.
\end{rem}
\begin{cor} \label{cor:moduli_object_equivalent_to_mirror}
  The correspondence $x_0 = u_1$, $x_i = v_i$ for $1 \leq i \leq n$, and $y_j = \rho_j$ for $1 \leq j \leq m$ defines an isomorphism of $\cM_{m,n} $ with $X_{n,m}$. \qed
\end{cor}
\begin{rem} \label{rem:proof_mirror_tori}
  Returning to Remark \ref{rem:discussion_mirror_tori}, recall that the Lagrangian $L_{m,n}^{i}$ is the product of $T^m$ with the Lagrangians obtained by performing Polterovich surgery on $L_{m,n}$ at $s$ for coordinates labelled by $0 \leq k \leq i$, and at $t$ for coordinates labelled by $i < k \leq n$. The surgery triangle \cite{FukayaOhOhtaOno2009} of Fukaya, Oh, Ohta and Ono identifies the isomorphism class in the Fukaya category of this surgery (with an appropriate local system) with the result of equipping $L_{m,n}$ with a bounding cochain with $u_0, \ldots, u_i \neq 0$, and $v_{i+1}, \ldots, v_n \neq 0$. Using the above correspondence, together with Equation \eqref{eq:equation-moduli-1}, this is the same as setting $x_0 x_1 \cdots x_k \neq 0$ for $k < i$, and $x_k \neq 0$ for $k > i$, which corresponds to setting all coordinates except for $x_i$ to be non-zero as asserted.
\end{rem}

We end this section by noting that the equivalence in Corollary \ref{cor:moduli_object_equivalent_to_mirror} is at this stage just an abstract isomorphism, and we do not know yet that it is induced by the mirror equivalence. We shall presently see that we have a natural map between these spaces. To construct it, we need the following basic Lemma:

\begin{lem}
  Let $R$ be the ring of functions on a smooth affine variety over a field $\bfk$, and let
  \begin{equation}
    M_+ \to M \to M_0    
  \end{equation}
be a short exact sequence of $A_\infty$ graded $R$-modules which are finite over $\bfk$, and with trivial differential. If $M_0$ is supported in degree $0$, and $M_+$ is supported in strictly positive degrees, then $M_0$ is a summand of $M$.
\end{lem}
\begin{proof}
 The module $M$ is classified by an element of $\Ext^{1}(M_0, M_+)$, but the fact that $\Ext$ groups of coherent sheaves are supported in non-negative degrees, and our support conditions imply that the only contribution comes from a map of sheaves, which would correspond to a differential in $M$.
\end{proof}

\begin{cor}
  The mirror functor defines a map
\begin{equation} \label{eq:mirror-map}
  \cM_{m,n} \to X_{n,m}
\end{equation}
which assigns to $b \in  \cM_{m,n} $ the support of the skyscraper sheaf associated to the degree $0$ component of the Floer group $HF^*\left((L_{m,n},b), L^+_{m,n}\right)$. \qed
\end{cor}

\subsection{The mirror map via symplectic cohomology}
\label{sec:mirror-map-via}

Our goal at this stage is to show that the mirror map in Equation \eqref{eq:mirror-map} is an isomorphism of varieties. Since the varieties are affine, it is easiest to state this in terms of rings of functions: let $\Gamma(X_{n,m})$ denote the ring of functions on $X_{n,m}$, which we know by homological mirror symmetry (i.e. the results of Section \ref{sec:homol-mirr-symm}) to be isomorphic to the self-Floer cohomology of $L^+_{m,n}$, as well as to the degree $0$ part of the symplectic cohomology $SH^0(X_{m,n})$. Let $\Gamma(\cM_{m,n})$ denote the ring of functions on $\cM_{m,n}$. Floer theory defines a closed-open map, that we call the \emph{mirror map},
\begin{equation} \label{eq:mirror_map-SH}
   SH^0(X_{m,n}) \to \Gamma(\cM_{m,n}).
\end{equation}

The fact that this agrees with the mirror map defined using Lagrangian Floer cohomology is given by the homotopy in Figure \ref{fig:mirror-map-pictures}.

\begin{figure}
  \centering
  \begin{tikzpicture}
    \begin{scope}
      \draw[thick] (-2,-.5) -- (2,-.5) ;
      \draw[thick] (-2,.5) -- (2,.5) ;
      \node[label = below:{$L_{m,n}$}] (Lmn) at (0,-.5) {};
      \node[label = above:{$L^+_{m,n}$}] (Lmn) at (0,.5) {};
      \node[label = right:{$SH^*(X_{m,n})$}] (SH) at (0,0) {};
      \filldraw[black] (SH) circle (2pt);
    \end{scope}
     \begin{scope}[shift = {(4.5,0)}]
      \draw[thick] (-2,-.5) -- (2,-.5) ;
      \draw[thick] (-2,.5) -- (2,.5) ;
        \node (SH) at (0,1) {};
      \draw[thick] (SH) circle (.5);
      \filldraw[black] (SH) circle (2pt);
      \draw[->] (0,.55) -- (0,.35);
      \node[label = below:{$L_{m,n}$}] (Lmn) at (0,-.5) {};
      \node[label = above:{$L^+_{m,n}$}] (Lmn) at (1,.5) {};
    \end{scope}
         \begin{scope}[shift = {(-4.5,0)}]
      \draw[thick] (-2,-.5) -- (2,-.5) ;
      \draw[thick] (-2,.5) -- (2,.5) ;
       \node (SH) at (0,-1) {};
      \draw[thick] (SH) circle (.5);
      \filldraw[black] (SH) circle (2pt);
      \draw[->] (0,-.55) -- (0,-.35);
      \node[label = below:{$L_{m,n}$}] (Lmn) at (1,-.5) {};
      \node[label = above:{$L^+_{m,n}$}] (Lmn) at (0,.5) {};
       \end{scope}
 \end{tikzpicture}
  \caption{The action of symplectic cohomology on Floer cohomology}
  \label{fig:mirror-map-pictures}
\end{figure}
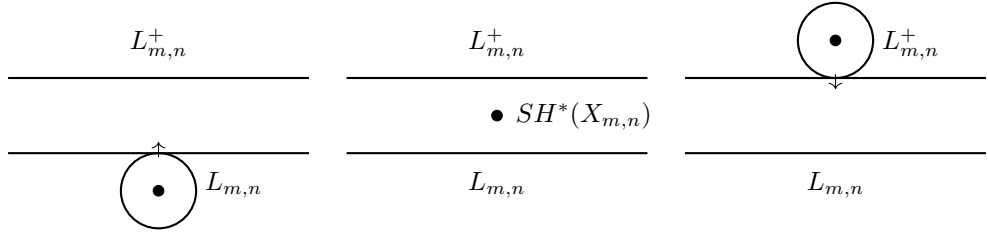

\begin{prop} \label{prop:mirror_map_iso}
  The mirror map in Equation \eqref{eq:mirror_map-SH} is an isomorphism.
\end{prop}
\begin{rem}
  Both the statement and the proof of this result only involve compact Lagrangians. This gives us flexibility in our choice of model for symplectic cohomology classes as we shall not need to construct the closed-open map for general (non-compact) Lagrangians, which usually involves more delicate arguments to ensure compactness of moduli spaces of holomorphic curves.
\end{rem}
Our proof of this result will be slightly roundabout: as discussed in Corollary \ref{cor:moduli_object_equivalent_to_mirror}, we have natural functions $u_1$,  $v_j$ for $1 \leq j \leq n$, and $\rho_i$ for $1 \leq i \leq m$, on $\cM_{m,n} $. 
The key technical result we shall prove is:
\begin{lem}\label{lem:lift_elements_to_SH}
  Each of the functions $u_1$, $v_j$ for $1 \leq j \leq n$, and $\rho^{\pm 1}_i$ for $0 \leq i \leq m$ lifts to $SH^0(X_{m,n})$ under the mirror map.
\end{lem}

Before proving this (lifting) result, we explain how it implies the main result of this section:
\begin{proof}[Proof of Proposition \ref{prop:mirror_map_iso}]
  We use the fact that $ SH^0(X_{m,n})$ and $ \Gamma(\cM_{m,n})$ are abstractly isomorphic Noetherian rings.
  Lemma \ref{lem:lift_elements_to_SH} implies that the mirror map is surjective. A surjective self-map of a Noetherian ring cannot have a non-trivial kernel, as the inverse images of the kernel under iterates of the map would give rise to an infinite ascending chain of ideals. We conclude that the kernel is trivial, hence that it is an isomorphism.

\end{proof}

\subsubsection{Construction of the lifts}
\label{sec:construction-lifts}

The purpose of this section is to prove Lemma \ref{lem:lift_elements_to_SH}; the proof will make use of the fact that, in a partial compactification $\bar{X}$ of a symplectic manifold $X$ by normal crossing divisors, whose total space is appropriately convex at infinity (e.g. which intersects each curve positively), each component of the divisor at infinity defines an element of $SH^*(X)$. The idea goes back to work of Seidel \cite{Seidel2002}, and has been repeatedly used in the literature, with the most complete current account given in \cite{GanatraPomerleano2021}. We shall in particular use the fact, proved in loc. cit., that given a compact Lagrangian $L \subset X$ underlying an object of the Fukaya category, the image of the element of $SH^*(X)$ associated to a divisor in $\bar{X}$ in the Floer homology of $L$, is given by the count of discs in $\bar{X}$ with boundary on $L$, whose intersection number with the divisor is $1$.

At this stage, we complete the proof of the existence of lifts of the desired elements to symplectic cohomology, which will feature a choice of a different partial compactification for each case:
\begin{proof}[Proof of Lemma \ref{lem:lift_elements_to_SH}]
  For the element $u_1$, we consider the space $\bar{X}_{m,n}^{u_1}$ which is the closure of $X_{m,n}$ in
  \begin{equation}
    \bC_{y_1}  \times \prod_{i=2}^{n} \bC^*_{y_i} \times \bC^{m+1}_{x},
  \end{equation}
  i.e. were we allow the variable $y_1$ to vanish. The maps $q_j$ extend to $\bar{X}_{m,n}^{u_1}$, and the divisor  $\bar{X}_{m,n}^{u_1} \setminus X_{m,n}$ is a copy of $X_{m,n-1}$ (i.e. given by the equation $\prod x_i = 1+ \sum y_j$, where the coordinates $y_j$ range from $2$ to $n$). The map $q_1$ is given by $1 + \frac{y_1}{y_2}$, so that the divisor at infinity projects to $1$ (and is in fact its entire inverse image). The count of holomorphic discs whose intersection number with this divisor is $1$ corresponds to holomorphic discs which in the $q_1$ factor have a simple $1$, and are constant in all other factors. With boundary conditions $L_{m,n}$,
  the only discs which could contribute to the closed-open map 
  are the two Pink discs, one of which is shown in the bottom left of Figure \ref{fig:Seidel-Lagrangian-1-and-3-gon}.
 Requiring this disc to pass through a chosen marked point on $L_{m,n}$ representing the unit on cohomology, we find a single contribution to this count, which is regular, and the fact that it has exactly one corner at $s$ shows that this divisor lifts the function $u_1$ to symplectic cohomology.

  We define the partial compactification $\bar{X}_{m,n}^{v_j}$ for $1 \leq j \leq n-1$ analogously by allowing the coordinate $y_{j+1}$ to vanish. The map $q_j$ again extends to $\bar{X}_{m,n}^{v_j}$, but now with target $\bC \bP^1$; the divisor $\bar{X}_{m,n}^{v_j} \setminus X_{m,n}$ maps to infinity. All other maps $q_{j'}$ for $j \neq j'$ are still well-defined with target $\bC$.  The count of discs which meet the divisor at infinity once now corresponds to holomorphic discs which in the $q_j$ factor pass through infinity once, and are constant in all other factors; requiring them to pass through a marked point representing $1$, there is a unique such disc, which moreover has one corner at $t$. This divisor thus lifts the function $v_j$ to symplectic cohomology.

  Next, we define the partial compactification $ \bar{X}_{m,n}^{v_n}$  by adding the product $\left(\bC^*\right)^m \times H_{n-1}$ as a fibre of $q_n$ over infinity; note that this construction makes sense as an algebraic variety because the complement of the fibres $q_n^{-1}(0)$ and $q^{-1}(1)$ is isomorphic to
      \begin{equation}
        \{ \prod x_i = 1\} \times \{ \sum y_j = -1 \}  \times \bC \setminus \{0,1\}
      \end{equation}
      with isomorphism over $\lambda \in \bC \setminus \{0,1\}$ given by scaling the variable $x_0$ by $\lambda$ and $(y_1, \ldots, y_j)$ by $1 - \lambda$.  The remaining functions $q_i$ also extend, because they are preserved by the trivialisation. The discs whose intersection number with $\bar{X}_{m,n}^{v_n} \setminus X_{m,n}$ is $1$ are again constant in all factors except the $q_n$ factor, in which they pass through infinity once; those with boundary conditions on $L_{m,n}$, and passing through the marked point associated to the unit 
      have one corner at $t$, hence lift $v_n$ to symplectic cohomology.

    In order to lift the functions $\rho^{-1}_i$ for $1 \leq i \leq m$, we consider the toric inclusion
    \begin{equation}
        \bC^{2}_{x_0, x_i} \subset \mathrm{Tot}_{\bC \bP^1}(\scrO(-1))    
    \end{equation}
    into the total space of the (toric) line bundle over the projective space $\bC \bP^1_{x_i}$, which over $\bC^1_{x_i}$ identifies the $x_0$ factor with the fibre. The map $\prod x_i$ on $\bC^{m} $ extends as a torically equivariant function from the product of this total space with $\bC^{m-2}$  to $\bC$, so we can define a partial compactification
    \begin{equation}
         \bar{X}_{m,n}^{\rho_i^{-1}} \subset    \mathrm{Tot}_{\bC \bP^1}(\scrO(-1))   \times  \bC^{m-1} \times \left( \bC^* \right)^{n} 
       \end{equation}
       of $X_{m,n}$ to be the hypersurface where the function equals $1 + \sum y_i$. The divisor at infinity is a copy of $X_{m-1,n}$, corresponding to the inclusion
       \begin{equation}
         \bC  \subset     \mathrm{Tot}_{\bC \bP^1}(\scrO(-1))
            \end{equation}
            as the fibre over  $x_i = \infty$. The maps $q_j$ are still well-defined, as are the maps $x_{i'}$ for $i \neq i'$. The projection map to the base $\bC \bP^1_{x_i}$ of the line bundle extends the projection map to $\bC_{x_i}$; the divisor at infinity maps to infinity.   The Lagrangian $L_{m,n}$ projects to the standard circle in this factor, which bounds one disc whose winding number around the boundary is the opposite of the usual orientation of the circle factors of the Clifford torus (given as the boundary of the disc in $\bC$), so we obtain a lift of $\rho_i^{-1}$ to symplectic cohomology.

            Finally, we consider the inclusion
            \begin{equation}
        \bC^{2}_{x_0,x_i} \subset \mathrm{Tot}_{\bC \bP^1}(\scrO(-1))    
      \end{equation}
      so that $x_0$ now corresponds to the coordinate on the base, and $x_i$ to the coordinate on the fibre away from infinity.  We obtain a partial compactification of $X_{m,n}$ which we denote $ \bar{X}_{m,n}^{\rho_i}$ exactly as above. The function $x_i$ extends as a toric map
      \begin{equation}
        \mathrm{Tot}_{\bC \bP^1}(\scrO(-1)) \to \bC
      \end{equation}
      which vanishes on the divisor at infinity, so that this map also extends to $ \bar{X}_{m,n}^{\rho_i} $, taking the divisor at infinity to $0$. We can then identify the count of discs in $\bC$ passing once through the origin with the count of discs passing through this divisor, which now lifts the function $\rho_i$ to symplectic cohomology.
  \end{proof}

\section{Homological Mirror Symmetry}
\label{sec:homol-mirr-symm}

The goal of this section is to prove that Homological mirror symmetry holds for the pairs $X_{m,n}$ and $X_{n,m}$ in the sense of Theorem \ref{thm:HMS}. Our strategy starts by noting that the total space of $X_{m,n}$ can be obtained by completing a domain obtained by gluing symplectic domains that yield $H_{n-1} \times \bC^{m+1}$ and $ \left(\bC^{*}\right)^{n +m}$ under completion, along the common hypersurface $H_{n-1} \times \left(\bC^{*}\right)^m $. Mirror statements for the partially wrapped Fukaya categories of the pairs $\left( \bC^{m+1}, \left(\bC^{*}\right)^{m}\right) $ and $\left( \left(\bC^{*}\right)^{n}, H_{n-1}\right) $ are by now well-established, and our proof will combine these results with various structural properties of Fukaya categories and categories of coherent sheaves.

More precisely, the following is the key technical result that we establish, where we write $\Pi_{n-1}$ for the union of the coordinate hyperplanes $ \prod_{i=0}^{n-1} x_i = 0 $ in $\bC^n$:
\begin{prop}
  \label{prop:comparison_restriction_functors}
  There is a commutative diagram of functors
\begin{equation} \label{eq:side-by-side-squares-restriction}
  \begin{tikzcd}[column sep=tiny]
\cF(H_{n-1} \times \bC^{m+1},H_{n-1} \times  \left(\bC^{*}\right)^m ) \ar[d]  &    \cF(H_{n-1} \times \left(\bC^{*}\right)^m ) \ar[r] \ar[d] \ar[l] &  \cF(\left(\bC^{*}\right)^{n +m}, H_{n-1} \times \left(\bC^{*}\right)^m)  \ar[d] \\
\Coh(\Pi_{n-1} \times H_{m-1}) &   \Coh(\Pi_{n-1} \times \left(\bC^{*}\right)^m) \ar[r] \ar[l] & \Coh(\bC^{n} \times \left(\bC^{*}\right)^m)
  \end{tikzcd}
\end{equation}
in which the vertical maps are equivalences, the top horizontal maps are cup (Orlov) functors, and the bottom vertical maps are respectively pullback and pushforward functors.
\end{prop}

In order to understand the relevance of the above result to Theorem \ref{thm:HMS}, we need a basic geometric construction:
\begin{lem}
  The projection map
  \begin{equation}
    X_{n,m} \to \bC^{n} \times \left(\bC^{*}\right)^m
  \end{equation}
  associated to forgetting the variable $x_0$ induces an isomorphism between $X_{n,m}$ and the complement of the proper transform of the hypersurface $\Pi_{n-1} \times \left(\bC^{*}\right)^m$ in the blowup of $ \bC^{n} \times \left(\bC^{*}\right)^m$ along $\Pi_{n-1} \times H_{m-1}$.
\end{lem}
\begin{proof}
  The codimension $2$ subvariety $\Pi_{n-1} \times H_{m-1}$ is the complete intersection associated to the two functions $f = \prod x_i $ and $g = 1 + \sum y_j$. The blowup is the subvariety of $\bC^{n} \times \left(\bC^{*}\right)^m \times \bP^1_{[X_0,Z]}$ given by
  \begin{equation}
    X_0 \cdot   \prod x_i = Z \cdot \left(   1 + \sum y_j \right).
  \end{equation}
  The proper transform of the subvariety $\Pi_{n-1} \times \left(\bC^{*}\right)^m $ is given in these equations by setting $Z=0$. The complement of this divisor has a standard coordinate $x_0$, with respect to which we obtain our equation for $X_{n,m}$ as a subvariety of $ \bC^{n+1} \times \left(\bC^{*}\right)^m$. 
\end{proof}

In the next part of this section, we explain how Proposition \ref{prop:comparison_restriction_functors} implies Theorem \ref{thm:HMS}. Afterwards, we explain the proof of the Proposition.

\subsection{Gluing coherent sheaves and Fukaya categories}
\label{sec:gluing-coher-sheav}

Let $X$ be a smooth variety and $Z \subset X$ a locally complete intersection subvariety of codimension $2$. Let $X_Z$ denote the blowup of $X$ along $Z$ (i.e. along the ideal of functions which vanish to first order at $Z$).  The key input for our computations is Orlov's formula \cite{Orlov1992} for the derived category of a blowup, which asserts that the derived category of the blowup has a semi-orthogonal decomposition
  \begin{equation}
    D^b \Coh X_{Z} \cong \langle D^b \Coh Z,  D^b \Coh X \rangle,
  \end{equation}
  where the embedding of $D^b \Coh Z $ in $D^b \Coh X_{Z}$ is given by composing the pull back to the exceptional divisor $E \subset X_Z$, with tensoring with the line bundle $\scrO(-1)$, and then with push forward to $X_Z$, while the embedding of $D^b \Coh X$ in $D^b \Coh X_{Z}$ is given by pullback along the projection. {Moreover, the defining bimodule is the pushforward map from $Z$ to $X$.}

  \begin{rem}
  Orlov's result was originally formulated only for smooth subvarieties, but is well-known among experts that it holds as well in the local complete intersection case (c.f. \cite[Theorem 3.4]{Kuznetsov2016}).
\end{rem}

  We now state a result which allows us to formulate a mirror to the blowup operation:
  \begin{figure}[h]
  \centering
  \begin{tikzpicture}

    \begin{scope}
      \node at (0,0) (F) {$F \times A_2$};
      \node at (-2,0) (X) {$E_1$};
      \node at (2,0) (Y) {$E_2$};
      \filldraw[black] (.75,0) circle (2pt);
      \filldraw[black] (-.75,0) circle (2pt);
      \filldraw[black] (0,-.75) circle (2pt);
      \draw[thick] (0,0) circle (.75);
      \draw[thick] (-2,0) circle (1.25);
       \draw[thick] (2,0) circle (1.25);
     \end{scope}
     
 \begin{scope}[shift={(6.5,0)}]
   \node at (0,0) (XFY) {$E_1 \cup_{F} \left(F \times A_2 \right) \cup_F E_2$};
   \draw[thick] (0,0) ellipse ( 2 and 1.25);
      \filldraw[black] (0,-1.25) circle (2pt);
     \end{scope}
  \end{tikzpicture}
  \label{fig:intermediate_gluing}
  \caption{The gluing procedure represented in the stopped formalism.}
\end{figure}
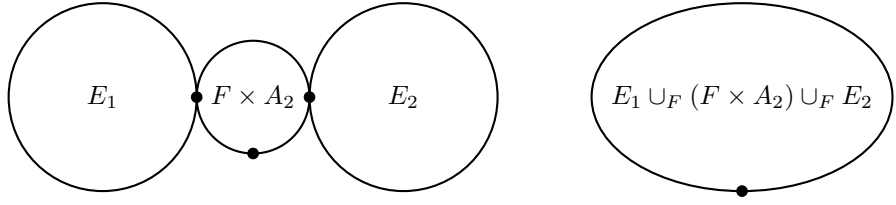

\begin{lem}[Corollary 3.11 of \cite{Sylvan2019}]
  \label{lem:sylvan-gluing}
  Let $E_1$ and $E_2$ be Liouville sectors with isomorphic Weinstein stop $F$. Let $E_1 \cup_{F} \left(F \times A_2\right) \cup_F E_2$ denote the Weinstein sector obtained by gluing $E_1$, $E_2$, and $F \times A_2$ along $F$. Then the Fukaya category of $E_1 \cup_{F} \left(F \times A_2\right) \cup_F E_2$ admits a semi-orthogonal decomposition
  \begin{equation} \label{eq:semi-orthogonal-A2-gluing}
    \langle \cF(E_2), \cF(E_1) \rangle,    
  \end{equation}
  with defining bimodule given by the graph of the composition
  \begin{equation}
 \cup_{E_1} \circ \cap_{E_2}^{L} \co    \cF(E_2) \to \cF(F) \to \cF(E_1)
  \end{equation}
  of the left cap functor for $E_2$, and the cup functor for $E_1$, which assigns to a pair $(L,K)$ of Lagrangians in $E_2$ and $E_1$ the Floer cohomology group
  \begin{equation}
    HF^*(\cup_{E_1} \circ \cap_{E_2}^{L} L,  K).    
  \end{equation}
\end{lem}
\begin{proof}
  Corollary 3.11 of \cite{Sylvan2019} gives the desired semi-orthogonal decomposition with gluing bimodule
  \begin{equation}
    \Gamma^\dagger(\cup_{E_1}) \otimes_{\cF(F)} \Gamma(\cup_{E_2}),
  \end{equation}
  where the graph bimodules above are given by $\Gamma^\dagger(\cup_{E_1})(K_{E_1}, L_F) = \hom_{E_1}(\cup_{E_1} L_F, K_{E_1})$ and $\Gamma(\cup_{E_2})(K_F, L_{E_2}) = \hom_{E_2}(L_{E_2}, \cup_{E_2} K_F)$. Using the adjunction $\Gamma(\cup_{E_2}) \cong \Gamma^\dagger(\cap_{E_2}^L)$, we get the desired formula
  \begin{align}
    \label{eq:gluing_bimodule_A_side}
    \Gamma^\dagger(\cup_{E_1}) \otimes_{\cF(F)} \Gamma(\cup_{E_2})
      &\cong \Gamma^\dagger(\cup_{E_1}) \otimes_{\cF(F)} \Gamma^\dagger(\cap_{E_2}^L) \\
      &\cong \Gamma^\dagger(\cup_{E_1} \circ \cap_{E_2}^L).
  \end{align}
\end{proof}
In order to see the relevance of this result for mirror symmetry, we return to the context of Orlov's decomposition, and consider a smooth hypersurface $H \subset X$ containing $Z$. The following result follows from the adjunction between restriction and pushforward, and the functoriality of restriction, which allows us to express the pushforward map from $Z$ to $X$ as a composition of the pushforward from $H$ to $X$, and the pushforward from $Z$ to $H$ which is the adjoint of restriction from $H$ to $Z$:
\begin{cor}\label{cor:HMS_blowup_via_gluing}
  There is a commutative diagram
  \begin{equation}
  \begin{tikzcd}
    D^b \Coh(Z) \ar[d] & D^b \Coh(H) \ar[d] \ar[l] \ar[r] &  D^b \Coh(X) \ar[d] \\
    \cF(E_1)   & \cF(F)\ar[l] \ar[r] & \cF(E_2)
  \end{tikzcd}
\end{equation}
in which the vertical map are equivalences, the top maps are given by pushforward and pullback, and the bottom maps are Orlov cup functors induces an equivalence
\begin{equation}
     D^b \Coh(X_Z) \cong  \cF(E_1 \cup_{F} \left(A_2 \times F\right) \cup_F E_2).
\end{equation} \qed
\end{cor}

In the example at hand, we are interested not in the blowup, but in the complement  $X_Z \setminus H$ of the proper transform of $H$ in $X_Z$.  The derived category of this complement is the quotient category
  \begin{equation}
    D^b \Coh ( X_Z \setminus H) \cong  D^b \Coh X_{Z}/ D^b \Coh H.
  \end{equation}
  In order to express the functor $D^b \Coh H \to D^b \Coh X_{Z} $ associated to the proper transform in terms of Orlov's semi-orthogonal decomposition, note that the functor associated to pushing forward from $H$ to $X$ then pulling back to $X_Z$ is represented by the composite $D^b \Coh H \to D^b \Coh X \to D^b \Coh X_{Z} $ corresponding to the second piece of the semi-orthogonal decomposition. This functor geometrically corresponds to the pullback from $H$ to its total transform. Since the total transform decomposes as the union of the exceptional divisor and the proper transform, we are led to consider the composite of the restriction functor $ D^b \Coh H \to D^b \Coh Z  $, with pullback from $Z$ to $E$, tensoring with the natural line bundle $\scrL$ on $E$, then pushing forward to $X_Z$, which corresponds to the first piece of Orlov's decomposition. The line bundle $\scrL$ arises from the projectivisation of the normal bundle of $Z$ in $X$, so the section defining $TH|_{Z}$ yields a natural transformation in the diagram of functors
  \begin{equation}
    \begin{tikzcd}
      D^b \Coh H \ar[r] \ar[d] & D^b \Coh Z \ar[d] \ar[dl,Rightarrow] \\
     D^b \Coh X \ar[r] &  D^b \Coh X_{Z}.
    \end{tikzcd}
  \end{equation}
{The cone of this natural transformation represents the pushforward map on the proper transform of $H$.} This can be expressed in terms of the semi-orthogonal decomposition as follows: { the adjunction between pullback and pushforward yields a natural transformation from the composition of restricting from $H$ to $Z$ and pushing forward from $Z$ to $H$, to the identity on $D^b \Coh H$. Composing with the pushforward from $H$ to $X$ then yields the desired natural transformation between the two functors from $D^b \Coh H$ to $\langle D^b \Coh Z,D^b \Coh X \rangle$  described above.} We now formulate the mirror Floer theoretic construction:

\begin{lem}[Proposition 3.12 of \cite{Sylvan2019}]
  \label{lem:GPS-gluing-from-Sylvan-gluing}
  With respect to the semi-orthogonal decomposition in Equation \eqref{eq:semi-orthogonal-A2-gluing}, the stop-removal functor
  \begin{equation}
       \cF(  E_1 \cup_{F} \left(A_2 \times F\right) \cup_F E_2) \to \cF( E_1 \cup_{F}E_2) 
     \end{equation}
     is obtained by localisation at the natural transformation $\cup_{E_1} \circ \eta_{E_2}$ (where $\eta_{E_2}  \co \cap_{E_2}^{L} \circ \cup_{E_2} \Rightarrow 1_{\cF(F)}$ is the co-unit of the cap-cup adjunction on $E_2$) between the two functors
     \begin{equation}
       \begin{tikzcd}
         \cF(F) \ar[r] \ar[d] &  \cF({E_1}) \ar[d] \\
         \cF(E_2) \ar[r] & \langle \cF(E_2), \cF(E_1) \rangle.    
       \end{tikzcd}
     \end{equation}
\end{lem}
\begin{proof}
  In the presentation $\Gamma^\dagger(\cup_{E_1}) \otimes_{\cF(F)} \Gamma(\cup_{E_2})$ of the gluing bimodule, Proposition 3.12 of \cite{Sylvan2019} asserts that stop removal is given by inverting a natural transformation which, at the level of objects, is given by $T_{L_F} = \id_{\cup_{E_1}L_F} \otimes \id_{\cup_{E_2}L_F}$. Since localisation only depends on the cohomology-level natural transformation, it is enough to show that the quasi-isomorphisms \eqref{eq:gluing_bimodule_A_side} sends $T_{L_F}$ to $\cup_{E_1}((\eta_{E_2})_{L_F})$. For this, we compute
  \begin{multline}
    T_{L_F} \mapsto \id_{\cup_{E_1} L_F} \otimes (\eta_{E_2})_{L_F} \\
    \mapsto \mu^2(\id_{\cup_{E_1} L_F},\cup_{E_1}
    \circ (\eta_{E_2})_{L_F}) = \cup_{E_1} \circ (\eta_{E_2})_{L_F}.
  \end{multline}
\end{proof}

At this stage, we are ready to deduce our homological mirror statement from the main result of this section
\begin{proof}[Proof of Theorem \ref{thm:HMS}]
Proposition \ref{prop:comparison_restriction_functors} asserts that the assumptions of Corollary \ref{cor:HMS_blowup_via_gluing} hold, yielding a comparison between the derived category of the blowup of $ \bC^{n} \times \left(\bC^{*}\right)^m$ along $\Pi_{n-1} \times H_{m-1}$ with the Fukaya category of $X_{m,n}$ stopped at $ H_{n-1} \times  \left(\bC^{*}\right)^m $. The description of the derived category of $X_Z \setminus H$ and of the Fukaya category of $X_{m,n}$ as localisations then yields the desired result.
\end{proof}

\begin{rem}
  A more efficient way to summarise the results of this section is to consider the diagram
    \begin{equation}
    \begin{tikzcd}
      D^b \Coh H \ar[r] \ar[d] & D^b \Coh Z \ar[d] \\
      D^b \Coh X \ar[r] & D^b \Coh ( X_Z \setminus H).
    \end{tikzcd}
  \end{equation}
  of categories. This can be shown to be a pushout diagram, and there are analogous results for the Fukaya categories of Liouville sections glued along a common stop (c.f.  Proposition 9.31 of \cite{GanatraPardonShende2019}). This formulation is used in \cite{GammageLe2021}.
\end{rem}

\subsection{A brief review of mirror symmetry for the three base cases}
\label{sec:brief-review-mirror}

We begin the proof of Proposition \ref{prop:comparison_restriction_functors}  by specifying our description of the three vertical maps in Diagram \eqref{eq:side-by-side-squares-restriction}.

One simplifying point is that the three varieties being considered on the coherent sheaf side are all affine, so that their derived categories of coherent sheaves are modelled by the category of complexes of free modules over the underlying ring of functions, with finitely generated total cohomology. We shall thus construct the desired functors by identifying Lagrangians in each of the three Liouville sections being considered, which are mirror to the structure sheaf. The discussion is further simplified by the intrinsic formality of algebras which are supported in degree $0$, i.e. they are determined up to $A_\infty$ equivalence by the cohomological product.

We begin by considering the functor
\begin{equation}
      \cF(H_{n-1} \times \left(\bC^{*}\right)^m ) \to  \Coh(\Pi_{n-1} \times \left(\bC^{*}\right)^m).
\end{equation}

On the symplectic side, we introduce the positive real Lagrangian
\begin{equation}
 L^+_{n-1}   \times \bR_+^m \subset H_{n-1} \times \left(\bC^{*}\right)^m,  
\end{equation}
i.e. the product of the set $ L^+_{n-1}$ of solutions to $\sum y_j = 1$ for $y_j \in \bR_+$ with the set $\bR_+^m $ of solutions to $\prod x_i = 1 $, again with $x_i \in \bR_+$.
\begin{lem} \label{lem:HMS-fibre}
  The self-Floer cohomology of $ L^+_{n-1}   \times \bR_+^m$ is isomorphic to the ring  of functions on $ \Pi_{n-1} \times \left(\bC^{*}\right)^m$, and the induced Yoneda functor
  \begin{equation}
    \cF(    H_{n-1} \times \left(\bC^{*}\right)^m) \to \mod_{\bk[x_j, y_i^{\pm}]/ \prod x_i }  
  \end{equation}
  is a fully faithful embedding onto the subcategory corresponding to the derived category of coherent sheaves.
\end{lem}
\begin{proof}
  By the K\"unneth formula \cite{GanatraPardonShende2019}, we have a commutative diagram
  \begin{equation}
    \begin{tikzcd}
      \cF(H_{n-1} \times (\bC^*)^m) \ar[r, "\cong"] \ar[d, "\cY_F"]
        & \cF(H_{n-1}) \otimes \cF((\bC^*)^m) \ar[d, "\mathcal Y_\mathrm{prod}"] \\
      \mod_{CF^*(L^{\bR}_{n-1} \times \bR_+^m)} \ar[r, "\cong"]
        & \mod_{CF^*(L_{n-1}^\bR)} \otimes \mod_{CF^*(\bR_+^m)} ,\\
    \end{tikzcd}
  \end{equation}
 in which both vertical arrows are given by testing against the object $L_{n-1}^\bR \times \bR_+^m$. It thus suffices to show that the functor $\cY_\mathrm{prod}$ is fully faithful, which amounts to showing that its two components are individually fully faithful. This is automatic for the second component, since $\bR_+^m$ generates $\cF( (\bC^*)^m )$, so it is enough to see that the first component
  \begin{equation}
    \cF(H_{n-1}) \to \mod_{CF^*(L_{n-1}^\bR)}
  \end{equation}
  is fully faithful.
  
  We shall appeal to \cite{LekiliPolishchuk2020}, in which it is proved that for a particular graded structure on $H_{n-1}$, there is an equivalence
  \begin{equation}
    \cF(H_{n-1}) \cong \Coh(\bk[x_1,\dotsc,x_n]/(x_1\dotsm x_n))
  \end{equation}
  sending a particular Lagrangian disc $L_{LP}$ to the rank $1$ free module. Using the fully faithful embedding $\Coh \hookrightarrow \mod$, this reduces our task to showing that $L_{n-1}^\bR$ is isotopic to $L_{LP}$ and that our grading choice agrees with theirs.

  We'll begin by recalling the construction of $L_{LP}$, which takes place in $H_{n-1}^{sym} = \Sym^{n-1}\Sigma$, where $\Sigma = \bC \setminus \{ p_1, \dotsc, p_n \}$. Varying the position of the points $p_i$ results in a deformation of Stein manifolds, so to construct a Liouville identification $H_{n-1}^{sym} \sim H_{n-1}$ we may use whichever $p_i$ we like. In this case, a convenient choice is to take
  \begin{equation}
    p_1 = 0 \ll - p_2 \ll \cdots \ll - p_{n-1} \ll 1 ,
  \end{equation}
  and $p_n=1$, where all $p_i$ are (negative) real and ``$\ll$'' is taken to make the diagonal approximation
  \begin{equation}
    \label{eq:sym_prod_determinant}
  \begin{vmatrix}
    1 & 0 & \cdots & 0 \\
    1 & p_2 & \cdots & p_2^{n-1} \\
    1 & p_3 & \cdots & p_3^{n-1} \\
    \vdots & \vdots & \ddots & \vdots \\
    1 & 1 & \cdots & 1
  \end{vmatrix}
  \approx p_2 p_3^2 \dotsm p_{n-1}^{n-2} \ne0
  \end{equation}
  hold. Note that, under the identification $ \Sym^{n-1}\bC \cong \bC^{n-1} $ given by
\begin{equation}
  \label{eq:sym_prod_identification}
  \begin{aligned}
    (a_1, \dotsc, a_{n-1}) &\mapsto \prod (x-a_i) \\
                           &= \sum z_{k+1} x^k \mapsto (z_1, \dotsc, z_{n-1}),
  \end{aligned}
\end{equation}
  $H_{n-1}^{sym}$ is sent to the complement of the affine hyperplanes
  \begin{equation}
    \label{eq:sym_prod_hyperplanes}
    \left\{p^{n-1}_{i} + \sum_{k=0}^{n-2} p_i^k z_{k+1} = 0\right\}_i,
  \end{equation}
  and the nonvanishing of the determinant in Equation \eqref{eq:sym_prod_determinant} asserts precisely that these hyperplanes are in general position.

  The Lagrangian $L_{LP} \subset H_{n-1}^{sym}$ is the product of the real arcs
  \begin{equation}
    (-\infty, p_{n-1}) \times (p_{n-1}, p_{n-2}) \times \dotsb \times (p_2, 0).
  \end{equation}
  Under the identification \eqref{eq:sym_prod_identification}, $L_{LP}$ lands in the positive real locus $\bR_+^{n-1}$.

  To identify $H_{n-1}^{sym}$ with our preferred presentation 
  \begin{equation}
    H_{n-1} = \bC^{n-1} \setminus \left\{1 + \sum z_k = 0\right\},
  \end{equation}
  we will deform the hyperplane loci \eqref{eq:sym_prod_hyperplanes} to the loci comprising the coordinate hyperplanes together with $\{1 + \sum z_k = 0\}$, or equivalently deform the matrix
  \begin{equation}
    M_0 =
    \begin{pmatrix}
      1 & 0 & \cdots & 0 & 0 \\
      1 & p_2 & \cdots & p_2^{n-2} & p_2^{n-1} \\
      \vdots & \vdots & \ddots & \vdots & \vdots\\
      1 & p_{n-1} & \cdots & p_{n-1}^{n-2} & p_{n-1}^{n-1} \\
      1 & 1 & \cdots & 1 & 1
    \end{pmatrix}
  \end{equation}
  to the matrix
  \begin{equation}
    M_1 =
    \begin{pmatrix}
      1 & 0 & \cdots & 0 & 0 \\
      0 & -1 & \cdots & 0 & 0 \\
      \vdots & \vdots & \ddots & \vdots & \vdots \\
      0 & 0 & \cdots & (-1)^{n-2} & 0 \\
      1 & 1 & \cdots & 1 & 1
    \end{pmatrix}
  \end{equation}
  through nonsingular matrices $M_t$. This we can do by convex interpolation, since we have chosen the $p_i$ to satisfy \eqref{eq:sym_prod_determinant}.

  Let us examine the effect of this deformation on the Lagrangian $L_{LP}$. Because all the coefficients are real, $L_{LP}$ is a connected component of $H_{n-1}^{sym} \cap \bR^{n-1}$. As we follow the deformation, we obtain a family of Lagrangians given by the corresponding connected component of the real locus. This family is induced by an isotopy because the hyperplane configuration that we are considering is generic.  To see that the image of $L_{LP}$ under this isotopy is $L_{n-1}^\bR = \bR_+^{n-1}$, it thus suffices to pick a point $z \in L_{LP}$, which belongs to $\bR_+^{n-1}$, and avoids the hyperplane configuration throughout the path. Choosing the point so that
  \begin{equation}
    1 \ll z_1 \ll \cdots \ll z_{n-1}    
  \end{equation}
  ensures that all the sign conditions on the components of $M\begin{pmatrix}z\\1\end{pmatrix}$ remain positive, so convex interpolation will preserve these inequalities, yielding the desired result.

  Finally, to see that our gradings agree with those of \cite{LekiliPolishchuk2020}, we note that the space of gradings is an affine space over the integral first cohomology of $H_{n-1}$, which has rank $n$. Since $L_{n-1}$ is contractible, we may identify the first homology of $H_{n-1}$ relative $L_{n-1}$ with the first homology of $H_{n-1}$. Inspecting the generators of the Floer complex of $H_{n-1}$ given in \cite{LekiliPolishchuk2020}, we see that the generators of the self-Floer cohomology of $L_{n-1}$ corresponding to the functions $x_i$ span the first homology. This means that any two grading structures on $H_{n-1}$ agree if the induced gradings on the self-Floer cohomology of $L_{n-1}$ are both supported in degree $0$ (c.f. \cite[Corollary 3.2.5]{LekiliPolishchuk2020}). This is the case for the gradings in \cite{LekiliPolishchuk2020} by their computation, and for our chosen gradings by noting that, if we compactify $H_{n-1}$ as the hyperplane $\sum_{i=0}^{n} Y_i = 0 $ in $\bC \bP^{n}$, the volume form in Equation \eqref{eq:holomorphic_volume_form} has a pole of order $1$ along the hypersurfaces $\{ Y_i = 0\}$ for $i \neq 0$ (but not for $Y_0 =0$). Since the closure of the Lagrangian $L_{n-1}$ in this compactification only meets the divisors along which the form has a pole of order $1$, we conclude that the Floer cohomology is supported in bounded degrees (c.f. \cite{Pomerleano2021}), hence necessarily in degree $0$ since any other choice of grading yields Floer groups supported in unbounded degrees.
\end{proof}

Next, we consider the left vertical arrow of Diagram \eqref{eq:side-by-side-squares-restriction}.  We associate to the potential
  \begin{equation}
   \prod_{i=0}^{m}x_i \co \bC^{m+1} \to \bC     
  \end{equation}
  a Liouville sector by restricting to the inverse image of the translate of the left half plane by a sufficiently large positive integer. The fibre of this Liouville sector is identified with $\left(\bC^* \right)^m$, and we simplify the notation by writing
  \begin{equation}
   \ell =  \bR_+^m \subset \left(\bC^* \right)^m.     
  \end{equation}
The next result is a computation of the self Floer cohomology of the image of this Lagrangian under the cup functor; since this computation is tied to the formulation of the cup functor itself, we include both in the statement:
  \begin{lem}[\cite{AbouzaidAuroux2021}] \label{lem:Abouzaid-Auroux}
    There is a commutative diagram of $A_\infty$ algebra maps
      \begin{equation}
      \begin{tikzcd}
 CF^*\left(\cup \ell \right)    \ar[d] &     CF^*(\ell) \ar[l] \ar[d] \\
        \bZ[y^{\pm}_1, \ldots, y_m^{\pm}]/ \left( 1 + \sum_{i=1}^{m} y_i \right)  & \bZ[y^{\pm}_1, \ldots, y_m^{\pm}]\ar[l]
      \end{tikzcd}
    \end{equation}
    in which the vertical arrows are quasi-isomorphisms.
  \end{lem}
  \begin{proof}[Sketch of proof:]
We start with some general considerations \cite{AbouzaidSmith2019,Sylvan2019a}: since the cup functor is spherical, there is a closed element $\tau_{\cup} \in CF^*(\ell)$ so that we have a commutative diagram
      \begin{equation}
      \begin{tikzcd}
        HF^*(\ell) \ar[r] \ar[dr] & HF^*\left(\cup \ell \right) \ar[d] \\
        & \mathrm{Cone}(  CF^*(\ell) \to  CF^*(\ell) ),
           \end{tikzcd}
    \end{equation}    
    where the morphism in the cone is multiplication by $\tau_{\cup}$, and the vertical map is an isomorphism. In our case, the Floer cohomology $HF^*(\ell) $ is supported in degree $0$, and is an algebra with no zero divisors, so that the above diagram simplifies to
     \begin{equation}
      \begin{tikzcd}
 HF^0\left(\cup \ell \right)\ar[dr] &        HF^0(\ell) \ar[l]  \ar[d] \\
        & HF^0(\ell)/ ( \tau_{\cup})
           \end{tikzcd}
         \end{equation}
at the cohomological level.
         
         In the formalism of \cite{Sylvan2019a} the element $\tau_{\cup}$ is characterised implicitly in terms of the existence of an adjoint to the cup functor. However, the map $HF^0(\ell) \to HF^0\left(\cup \ell \right)  $ is explicitly given as the isomorphism between the self-Floer cohomology of $\ell$ in the fibre and the self-Floer cohomology of its product with a cotangent fibre in a neighbourhood of the boundary modelled after $F \times T^* [0,1]$, followed by the continuation map from local to global wrapping. This isomorphism is natural, which  allows us to replace $\ell$ (which we recall is the positive real locus in our situation) with the compact Lagrangian $\bT^{m} \subset \left(\bC^* \right)^m$, which corresponds to the $0$-section of the cotangent bundle under the symplectomorphism $\left(\bC^* \right)^m \cong  T^* \bT^{m}$. The reason to pass to a compact Lagrangian is to ensure that conical Lagrangians can be produced by parallel transport along arcs in the base of our chosen Landau-Giznburg model.

         More precisely, we replace the cotangent fibre by the $0$-section, equipped with the local system corresponding to the group ring (c.f.  \cite{Abouzaid2011,Abouzaid2012a}). These objects of the Fukaya category are not isomorphic, but the Floer homology from $\ell$ to this local system is a free module of rank-$1$ over $HF^*(\ell, \ell)$; choosing the intersection point between the fibre and the $0$-section to be the base point       fixes the isomorphism
    \begin{equation}
         HF^*(\ell) \cong   \bZ[y^{\pm}_1, \ldots, y_m^{\pm}].   
    \end{equation}

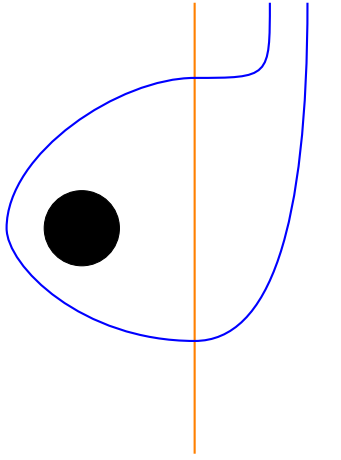
\begin{figure}[h]
  \centering
  \begin{tikzpicture}
    \draw[thick] (0,-3) -- (0,3);
\filldraw[black]  (-3.5,0) circle (.5);
\draw[orange, thick]  (-2 ,-3) -- (-2,3);
  \draw[blue,thick] (-1,3) .. controls (-1,2)  .. (-2,2) .. controls (-3,2) and (-4.5,1) ..  (-4.5,0) .. controls  (-4.5,-0.5) and (-3.5,-1.5) .. (-2,-1.5) .. controls (-1,-1.5) and (-0.5,0) .. (-0.5,3);

\end{tikzpicture} 
  \caption{Computing the cup functor on fibrewise compact Lagrangians in Liouville sectors associated to Landau-Ginzburg models: the endomorphisms of the image of the cup functor are given by the Floer homology groups of the orange and blue Lagrangians. The black circle represents the region where there may be some critical points.}
  \label{fig:Orlov-functor}
\end{figure}

The above considerations reduce the problem to computing the element $ \tau_{\cup} $ for the torus, equipped with its universal local system, which we do indirectly. We can decompose the contributions by intersection number with the components of the fibre: since $\tau_{\cup}$ is a count of strips as in Figure \ref{fig:Orlov-functor}, and the fibre at $0$ consists of complex hypersurfaces, we can decompose $\tau_{\cup}$ as a sum of contributions indexed by the components of this fibre. The count of curves passing through the irreducible component $\{x_i = 0\}$ can be identified with a count of discs in the product of $ \bC_{x_i} $ with $\bC^*_{x_j}$ for $j \neq i$. Note that the restriction of the potential function $\prod_{i} x_i$ to this open subset is a holomorphically trivial fibration with fibre $\left( \bC^*\right)^{m}$. Deforming the symplectic form to the standard one, we see that the count of discs in this relative homotopy class is $1$, hence that the element $\tau_{\cup} $ is a sum of $m+1$ terms; the expression $1 + \sum_{i=1}^{m} y_i$ arises by observing that the different trivialisation differ by their winding number (we use the component $\{x_0=0\}$ as the reference component, in which case $y_i$ is the winding number of $x_i/x_0$).

Finally, having computed the cup functor map on cohomology, we note all the higher terms of an $A_\infty$ homomorphism of graded $A_\infty$ algebras that are supported in degree $0$ must vanish for degree reasons. This implies that our cohomological computation yields the desired cochain level statement.
  \end{proof}

\begin{cor} \label{cor:HMS-left-vertical-arrow}
  The self-Floer cohomology of the Lagrangian $\cup  \left( L^+_{n-1}   \times \bR_+^m \right)$, as an object of the Fukaya category of $H_{n-1} \times \bC^{m+1} $  stopped at  $H_{n-1} \times  \left(\bC^{*}\right)^m $, is isomorphic to the ring  of functions on $ \Pi_{n-1} \times H_{m-1}$, and the induced Yoneda functor
  \begin{equation}
    \cF(    H_{n-1} \times \left(\bC^{*}\right)^m) \to \mod_{\bk[x_j, y_i^{\pm}]/(\prod x_j, 1 + \sum y_i)}  
  \end{equation}
  is a fully faithful embedding onto the subcategory corresponding to the derived category of coherent sheaves.
\end{cor}
\begin{proof}
  Arguing as in the proof of Lemma \ref{lem:HMS-fibre}, we can appeal to the K\"unneth theorem and reduce the problem to showing that the two component Yoneda functors are fully faithful. The first factor is the same as in Lemma \ref{lem:HMS-fibre}, and using the above result reduces the problem for the second factor to showing that $\cup\ell$ split-generates the Fukaya category. This follows from stop removal \cite{Sylvan2019, GanatraPardonShende2019}, the fact that $\ell$ generates the Fukaya category of the fibre, and the fact that $\cF(\bC^{m+1})$ is trivial.
\end{proof}

Finally, we consider the other side of this diagram. The next result can be extracted from work of Hanlon and Hicks (c.f. \cite[Theorem 6.2]{HanlonHicks2021}).
  \begin{lem}
    \label{lem:positive_real_locus_generates}
    The self-Floer cohomology of the positive real locus of $\left(\bC^{*}\right)^{n+m}$, stopped at $H_{n-1} \times \left(\bC^{*}\right)^m$  is isomorphic to the ring  of functions on $ \bC^{n} \times \left(\bC^{*}\right)^m$. This Lagrangian generates this partially wrapped category, so that the induced Yoneda functor
  \begin{equation}
 \cF(\left(\bC^{*}\right)^{n} \times \left(\bC^{*}\right)^m, H_{n-1} \times \left(\bC^{*}\right)^m)  \to \mod_{\bk[x_j, y_i^{\pm}]/ 1 + \sum y_i }  
  \end{equation}
  is a fully faithful embedding onto the subcategory corresponding to the derived category of coherent sheaves.
  \end{lem}
  \begin{proof}
    The positive real locus is the unique handle for a Weinstein presentation both of $\bC^*$, and of the Weinstein sector corresponding to the potential $x$ on $\bC^*$.     By Gammage-Shende \cite{GammageShende2017}, the sector associated to a product potential is isomorphic to the product of the sectors corresponding to each factor. In particular, the product of the positive real loci is the unique handle for a Weinstein presentation of the pair $(\left(\bC^{*}\right)^{n} \times \left(\bC^{*}\right)^m, H_{n-1} \times \left(\bC^{*}\right)^m) $, hence generates the corresponding partially wrapped category. The computation of the self-Floer cohomology then follows from the K\"unneth formula.
  \end{proof}

  \begin{rem} In the above proof, we use the fact that the co-cores of a Weinstein sector generates its Fukaya category. The statement that appears in the literature (c.f.  \cite[Theorem 1.2]{ChantraineDimitroglouGhigginiGolovko2019} and \cite[Theorem 1.10]{GanatraPardonShende2019}), is that co-cores generate together with linking discs of the Legendrian stop. This distinction can be eliminated by choosing the Weinstein structure near the boundary to have a local minimum, rather than to be modelled after a cotangent bundle. Figure \ref{fig:linking_discs_are_cocores} shows how to make this replacement for a given Weinstein presentation.  To conclude that the co-cores of an arbitrary Weinstein structure which is modelled near the boundary by the right side of Figure \ref{fig:linking_discs_are_cocores}, we can appeal to \cite[Proposition 1.29]{GanatraPardonShende2019} and \cite[Proposition 3.3]{Lazarev2019} which together imply that the category generated by the co-cores does not depend on the choice of Weinstein presentation.
  \end{rem}
  \begin{figure}[h]
  \centering
  \begin{tikzpicture}
    \begin{scope}
    \draw[thick] (0,-2) -- (0,2);
    \draw[thin] (-2,0) -- (0,0);
    \foreach \i in {0, ..., 4}
      {
    \draw[->, thin] (-\i*.5,1) -- ++ (90:1);
    \draw[->, thin] (-\i*.5,-1) -- ++ (-90:1);
    \draw[->, thin] (-\i*.5,.2) -- ++ (90:.2);
    \draw[->, thin] (-\i*.5,-.2) -- ++ (-90:.2);
  };

  \draw[orange, thick]  (-1.5,-2) -- (-1.5,2);
    \end{scope}
    
  \begin{scope}[shift={(5,0)}]
    \draw[thick] (0,-2) -- (0,2);
    \draw[->, thin] (-1,0) -- ++ (180:.2);
    \draw[->, thin] (-.5,0) -- ++ (180:.2);
    \draw[->, thin] (-2,0) -- ++ (0:.2);
 \draw[->, thin] (-.5,1) -- ++ (105:1);
 \draw[->, thin] (-.5,-1) -- ++ (-105:1);
  \draw[->, thin] (-.5,.2) -- ++ (135:.2);
    \draw[->, thin] (-.5,-.2) -- ++ (-135:.2);
    
    \foreach \i in {0, 3}
      {
    \draw[->, thin] (-\i*.5,1) -- ++ (90:1);
    \draw[->, thin] (-\i*.5,-1) -- ++ (-90:1);
    \draw[->, thin] (-\i*.5,.2) -- ++ (90:.2);
    \draw[->, thin] (-\i*.5,-.2) -- ++ (-90:.2);
  };
  \foreach \i in {-1, 1}
  {
  
    \draw[->, thin] (-1.5+\i*.5,1) -- ++ (90+\i*15:1);
    \draw[->, thin] (-1.5+-\i*.5,-1) -- ++ (-90+\i*15:1);
    \draw[->, thin] (-1.5+-\i*.5,.2) -- ++ (90-\i*45:.2);
    \draw[->, thin] (-1.5+-\i*.5,-.2) -- ++ (-90+\i*45:.2);
  };

  \draw[orange, thick]  (-1.5,-2) -- (-1.5,2);
  \filldraw (-1.5,0) circle (2pt) ;
  \filldraw (0,0) circle (2pt) ;
    \end{scope}
  \end{tikzpicture}
  \caption{The vertical (orange) line represents a linking disc on the left, and the corresponding co-core on the right.}
   \label{fig:linking_discs_are_cocores}
\end{figure}
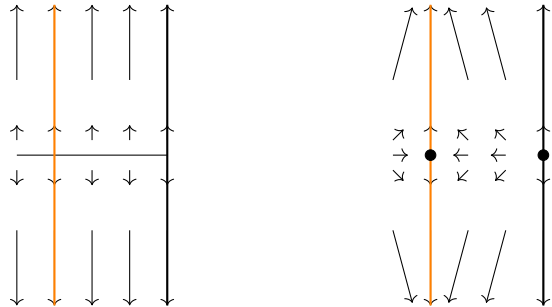

\subsection{Compatibility of mirror symmetry with restriction}
\label{sec:comp-mirr-symm-1}

Our goal in this section is to explain the commutativity of the left square in Diagram \eqref{eq:side-by-side-squares-restriction}. As a starting point, recall that the pullback map on derived categories of coherent sheaves of affine varieties associated to a ring map $R \to S$ can be expressed on (complexes of) modules as a tensor product with $S$, considered as an $R-S$-bimodule. Lemma \ref{lem:HMS-fibre} and Corollary \ref{cor:HMS-left-vertical-arrow} thus reduce the problem to proving the commutativity of the diagram
      \begin{equation}\label{diag:cup_and_Koszul}
  \begin{tikzcd}[column sep=tiny]
\cF(H_{n-1} \times \bC^{m+1},H_{n-1} \times  \left(\bC^{*}\right)^m ) \ar[d]  &    \cF(H_{n-1} \times \left(\bC^{*}\right)^m ) \ar[l] \ar[d] \\ \mod_{CF^*(\cup \left(L^{\bR}_{n-1} \times \bR_+^m \right) ) }  & \ar[l]\mod_{CF^*(L^{\bR}_{n-1} \times \bR_+^m)  }  \\
  \end{tikzcd}
\end{equation}
where the bottom horizontal map is tensoring with $CF^*(\cup L^{\bR}_{n-1} \times \bR_+^m) ) $ over $CF^*(L^{\bR}_{n-1} \times \bR_+^m) $, as long as we can compute that this map agrees with the mirror map:
\begin{multline}
     \bZ[x_1, \ldots, x_n, y^{\pm}_1, \ldots, y_m^{\pm}]/\left( \prod x_j \right)  \to \\ \bZ[x_1, \ldots, x_n, y^{\pm}_1, \ldots, y_m^{\pm}]/ \left( \prod x_j,  1 + \sum_{i=1}^{m} y_i \right).
\end{multline}
Lemma \ref{lem:Abouzaid-Auroux} computes this map if we ignore the $H_{n-1} $ factor. In order to obtain the desired computation, we therefore need a compatibility statement between cup functors associated to Liouville sectors and K\"unneth theorems for wrapped Floer cohomology. Because of intrinsic formality,  the required statement are at the cohomological level, which greatly simplifies the proof. For the statement, let $M$ be a Liouville sector with fibre $F$, and let $\ell \subset F$ be a properly embedded exact Lagrangian. We write $\cup \ell \subset M$ for the image of $\ell$ under the (Orlov) cup functor. Let $L \subset H$ be an exact Lagrangian in a Liouville domain. The following result is proved in Appendix \ref{sec:comp-k:unn-maps}:
  \begin{lem}
    The cup functors fit in a commutative diagram
    \begin{equation}
      \begin{tikzcd}
        HF^*(\ell \times L) \ar[r] \ar[d] & HF^*\left(\cup (\ell \times L) \right) \ar[d] \\
        HF^*(\ell) \otimes HF^*(L) \ar[r] & HF^*\left(\cup \ell\right)  \otimes HF^*(L).
      \end{tikzcd}
    \end{equation}
  \end{lem}

We shall now prove the commutativity of Diagram \eqref{diag:cup_and_Koszul}:
\begin{lem}
  \label{lem:left-side}
  The left square in Diagram \eqref{eq:side-by-side-squares-restriction} commutes.
\end{lem}
  \begin{proof}
   Write $L$ for the object $L_{n-1}^\bR \times \bR_+^m \in \cF(H_{n-1} \times (\bC^*)^m)$, let $K$ be an arbitrary object of $\cF(H_{n-1} \times (\bC^*)^m)$, and write $\tilde K$ for its image in $\mod_{CF^*(L_{n-1}^\bR \times \bR_+^m)}$. Write $\tilde K = \colim_D \tilde L$, where $L = L_{n-1}^\bR \times \bR_+^m$ and $\tilde L$ is $L$ as a module over itself. We may arrange that $D$ is a finite diagram of filtered diagrams. Now, because the right vertical arrow of Diagram \eqref{diag:cup_and_Koszul} is fully faithful, the colimit diagram $D_+ \to \mod_{CF^*(L_{n-1}^\bR \times \bR_+^m)}$ lifts to a diagram $D_+ \to \cF(H_{n-1} \times (\bC^*)^m)$, and the lift is a colimit diagram because it satisfies the universal property. In other words, $K$ is a finite colimit of filtered colimits of $L$, and this colimit is preserved by the right arrow.

    Next, note that both horizontal arrows are spherical functors \cite{AbouzaidSmith2019,Sylvan2019a}, so in particular they are left adjoints. This means they preserve the colimits $K = \colim_{D} L$ and $\tilde K = \colim_D \tilde L$. Using the fact that the left vertical map is a fully faithful embedding as above, we see that all arrows preserve the colimit $K = \colim_D L$.

    Now, let $K'$ be another object of $\cF(H_{n-1} \times (\bC^*)^m)$. Following the above, we can write $K' = \colim_{D'} L$, and this colimit is also preserved by all arrows in the diagram. We are thus interested in seeing that \eqref{diag:cup_and_Koszul} commutes for the morphism space $\hom(K, K')$. Since
    \begin{equation}
      \hom(K, K') = \hom(\colim_D L, \colim_{D'} L) = \lim_D \hom(L, \colim_{D'} L),
    \end{equation}
    and similarly in the other three categories, it is enough to show that the diagram commutes for morphism spaces of the form $\hom(L, \colim_{D'} L)$. To do this, we would again like to pull the colimit out of the $\hom$ complex. In the lower (module) categories that is possible because $\tilde L$ and its image in the lower left are compact objects. In the upper (Fukaya) categories, we use the tautological maps of cochain complexes
    \begin{equation}
      \label{eq:L_commute_colimit}
      \colim_{D'} \hom(L, L) \to \hom(L, \colim_{D'} L)
    \end{equation}
    (and similarly with $\cup L$) which become quasi-isomorphisms under the vertical arrows. On the other hand, the vertical arrows are fully faithful, so the map \eqref{eq:L_commute_colimit} is already a quasi-isomorphism. This reduces commutativity of Diagram \eqref{diag:cup_and_Koszul} to its commutativity on $L$, which is the Yoneda lemma.
  \end{proof}

\subsection{Compatibility of mirror symmetry with pushforward}
\label{sec:comp-mirr-symm}
We now turn our attention to the right square in Diagram \eqref{eq:side-by-side-squares-restriction}. Since the vertical maps in this diagram are equivalences, it suffices to prove that the diagram
  \begin{equation} \label{eq:adjoint_of_right_square}
  \begin{tikzcd}
  \cF(H_{n-1} \times \left(\bC^{*}\right)^m )  \ar[d]  &  \cF(\left(\bC^{*}\right)^{n} \times \left(\bC^{*}\right)^m, H_{n-1} \times \left(\bC^{*}\right)^m)  \ar[d] \ar[l] \\
 \Coh(\{\prod_{i=0}^{n-1} x_i = 0 \} \times \left(\bC^{*}\right)^m) & \Coh(\bC^{n} \times \left(\bC^{*}\right)^m) \ar[l],
  \end{tikzcd}
\end{equation}
given by the adjoint maps commutes. On the coherent sheaf side, the map is again given by restriction, so our proximate goal is to describe the adjoint of the cup functor on Fukaya categories.

\begin{lem}
  \label{lem:describe_adjoint_cup_as_restriction}
  If $M$ is a Liouville sector with fibre $F$, which is swappable in the sense of \cite{Sylvan2019a}, the adjoint to the cup functor $\cF(F) \to \cF(M)$ is given by a composition
  \begin{equation}
       \cF(M) \to   \cF(M^{(2)}) \to \cF(A_1 \times F) \cong \cF(F),
     \end{equation}
     where the first arrow is the inclusion of $M$ in a Liouville section $M^{(2)}$ obtained by doubling the stop, and the second arrow is a Viterbo restriction functor.
\end{lem}
\begin{proof}
  We start with the graph bimodule $\Gamma(\cup)(\ell, L) = CW^*(L, \cup \ell)$, illustrated as the first picture of Figure \ref{fig:restriction_to_fibre-is-viterbo}. We will apply a sequence of bimodule quasi-isomorphisms to obtain a formula for the adjoint graph of the left adjoint $\cap^L$.
 \begin{figure}[h]
  \centering
  \begin{tikzpicture}
    \begin{scope}
    \draw[thick] (0,-2) -- (0,2);
\draw [dashed, thin] (0,-2)  -- (-3,-2) arc (270:90:2) -- (0,2);
    \draw[orange, thick]  (-1,-2) -- (-1,2);
    \node[label=right:{$\cup \ell$}] at (-1,0) (Ul) {};

    \draw [dashed, thin] (-3.75,0) circle (.5);
    \fill[blue!50]   (-3,1) .. controls (-2,1) and (-2,-1) ..  (-3,-1) .. controls (-3.5,-1) and ($(-3.75,0) + (-60:.75)$) .. ($(-3.75,0) + (-60:.5)$) arc (-60:-120:.5) ..  controls  ($(-3.75,0) + (-120:1)$) and (-4,-1.5) .. (-3, -1.5) .. controls (-2.5,-1.5) .. (-2,-2) -- (-1.5,-2) -- (-1.5,2) -- (-2,2) .. controls (-2.5,1.5) .. (-3,1.5) .. controls (-4,1.5) and ($(-3.75,0) + (120:1)$) .. ($(-3.75,0) + (120:.5)$) arc (120:60:.5) .. controls ($(-3.75,0) + (60:.75)$) and (-3.5,1) ..  cycle;
    \node[label=left:{$L$}] at (-2.5,0) (L) {};  
    \end{scope}
     \begin{scope}[shift={(7,0)}]
       \draw [dashed, thin] (0,-2)  -- (-3,-2) arc (270:90:2) -- (0,2);
       \draw[thick] (0,-2) -- (0,-.5);
       \draw[dashed,thin] (0,.5) arc (90:270:.5); 
\draw[thick] (0,2) -- (0,.5);

    \draw[orange, thick]  (-.25,-2) -- (240:.5);
   \node[label=left:{$\cup_- \ell$}] at (-.25,-1) (Ul+) {};
   \draw [dashed, thin] (-3.75,0) circle (.5);
   \fill[blue!50]   (-3,1) .. controls (-2,1) and (-2,-1) ..  (-3,-1) .. controls (-3.5,-1) and ($(-3.75,0) + (-60:.75)$) .. ($(-3.75,0) + (-60:.5)$) arc (-60:-120:.5) ..  controls  ($(-3.75,0) + (-120:1)$) and (-4,-1.5) .. (-3, -1.5) .. controls (-2.5,-1.5) .. (-2,-2) -- (-1.5,-2) -- (-1.5,2) -- (-2,2) .. controls (-2.5,1.5) .. (-3,1.5) .. controls (-4,1.5) and ($(-3.75,0) + (120:1)$) .. ($(-3.75,0) + (120:.5)$) arc (120:60:.5) .. controls ($(-3.75,0) + (60:.75)$) and (-3.5,1) ..  cycle;
 
\end{scope}

\begin{scope}[shift={(0,- 5)}]
  \draw [dashed, thin] (0,-2)  -- (-3,-2) arc (270:90:2) -- (0,2);
      \draw[dashed,thin] (0,.5) arc (90:270:.5); 
       \draw[thick] (0,-2) -- (0,-.5);

\draw[thick] (0,2) -- (0,.5);
       
 \draw[blue!50, thick]  (-.25,2) -- (120:.5);

    \draw[orange, thick]  (-.25,-2) -- (240:.5);
  
       \draw [dashed, thin] (-3.75,0) circle (.5);
       \fill[blue!50] (-3,1) .. controls (-2,1) and (-2,-1) ..  (-3,-1) .. controls (-3.5,-1) and ($(-3.75,0) + (-60:.75)$) .. ($(-3.75,0) + (-60:.5)$) arc (-60:-120:.5) ..  controls  ($(-3.75,0) + (-120:1)$) and (-4,-1.5) .. (-3, -1.5) .. controls (-2,-1.5) and (210:1) ..  (210:.5) arc (210:150:.5) .. controls (150:1) and (-2,1.5)    .. (-3,1.5) .. controls (-4,1.5) and ($(-3.75,0) + (120:1)$) .. ($(-3.75,0) + (120:.5)$) arc (120:60:.5) .. controls ($(-3.75,0) + (60:.75)$) and (-3.5,1) ..  cycle;

  \end{scope}
  \begin{scope}[shift={(7,- 5)}]
  \draw [dashed, thin] (0,-2)  -- (-3,-2) arc (270:90:2) -- (0,2);
      \draw[dashed,thin] (0,.5) arc (90:270:.5); 
       \draw[thick] (0,-2) -- (0,-.5);

\draw[thick] (0,2) -- (0,.5);
       
 \draw[orange, thick]  (-.25,2) -- (120:.5);
 \node[label=left:{$\mathbf S_\phi \cup_+ \ell$}] at (-.25,-1) (Ul+) {};
 
    \draw[blue!50, thick]  (-.25,-2) -- (240:.5);

    \draw [dashed, thin] (-3.75,0) circle (.5);
         \fill[blue!50] ($(-3,0) + (165:2)$) arc (165:195:2)  .. controls ($(-3,0) + (195:1.5)$) and  ($(-3.75,0) + (-60:2.5)$) .. ($(-3.75,0) + (-60:.5)$) arc (-60:-120:.5) ..  controls  ($(-3.75,0) + (-120:1.5)$) and  ($(-3.75,0) + (120:1.5)$)  .. ($(-3.75,0) + (120:.5)$) arc (120:60:.5) .. controls ($(-3.75,0) + (60:2.5)$) and ($(-3,0) + (165:1.5)$) .. ($(-3,0) + (165:2)$) -- cycle ;

  \end{scope}
  \end{tikzpicture}
  \caption{The vertical (orange) line represents the cup functor. The (light blue) region in the top left is the projection of an arbitrary Lagrangian $L$ underlying an object in $\cF(M)$.}
  \label{fig:restriction_to_fibre-is-viterbo}
\end{figure}
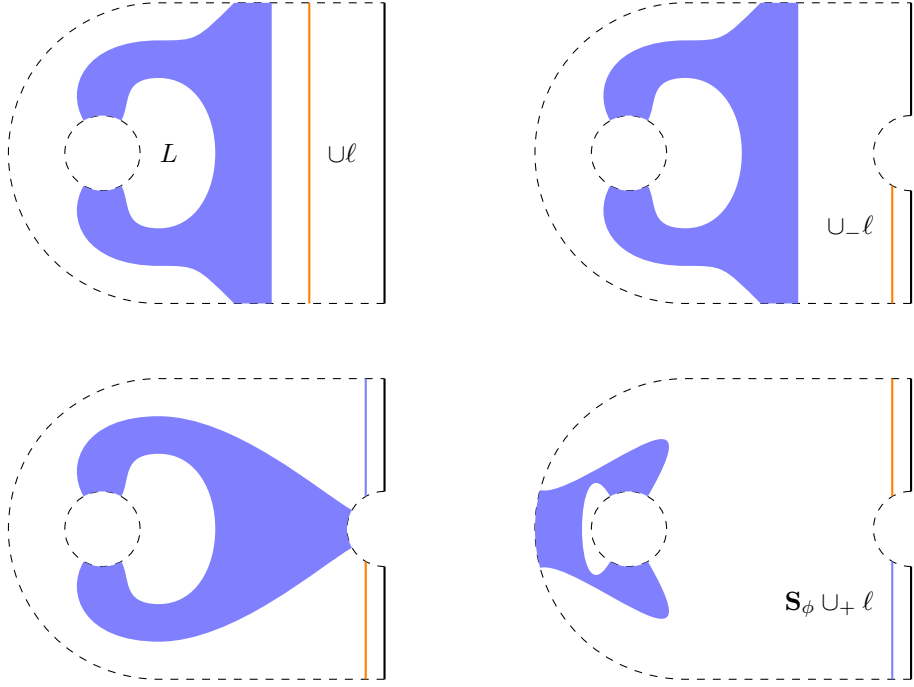
  First, we double the stop $\sigma$ to obtain a new sector $M^{(2)}$ with stops $\sigma^+$ and $\sigma^-$, which are positive and negative pushoffs of $\sigma$, respectively. In purely sectorial language, this is equivalent to gluing on a copy of $F \times A_2$, where $F$ is the fiber of $\sigma$ and $A_2$ is $\bC$ with three stops. By \cite[Cor.~3.11, see also proof of Thm.~1.3]{Sylvan2019a}, we have an equivalence
  \begin{equation}
    \Gamma(\cup) \cong (\cup_{\sigma_-}, i_M )^*\Delta_{\cF(M^{(2)})},
  \end{equation}
  where $i_M$ is the inclusion of $M$ into $M^{(2)}$ given by the gluing decomposition. This situation is illustrated in the second picture (upper right) of Figure \ref{fig:restriction_to_fibre-is-viterbo}.

  Next, by sphericality of $\cup$, we have a semiorthogonal decomposition
\begin{equation}
  \label{eq:cup_semiorthogonal_decomp}
    \cF(M^{(2)}) = \langle \cF(M), \cF(F) \rangle
\end{equation}
  with $\cF(F)$ the image of $\cup_{\sigma^+}$ and $\cF(M)$ the second orthogonal to the image of $i_M$. We are free to represent $i_M$ in these coordinates, see the third picture of Figure \ref{fig:restriction_to_fibre-is-viterbo}.

  Finally, we apply the wrap-once deformation $\phi \colon M^{(2)} \to M^{(2)}$ coming from the swap of $\sigma$, as shown in the last picture of Figure \ref{fig:restriction_to_fibre-is-viterbo}. Algebraically, this corresponds to taking second orthogonals, which means that we have an equivalence of bimodules
  \begin{align*}
    (\cup_{\sigma_-}, i_M)^*\Delta_{\cF(M^{(2)})}
      &\cong (\mathbf S_\phi \circ \cup_{\sigma_-}, \mathbf S_\phi \circ  i_M)^*\Delta_{\cF(M^{(2)})} \\
      &\cong (\cup_{\sigma^+}, \mathbf S_\phi \circ  i_M)^*\Delta_{\cF(M^{(2)})},
  \end{align*}
  where $\mathbf S_\phi$ is the autoequivalence of $\cF(M^{(2)})$ induced by $\phi$. Here, the positivity of the wrapping $\phi$ is expressed by the isomorphism $\mathbf S_\phi\circ \cup_{\sigma_-} \cong \cup_{\sigma_+}$; a negative wrapping would involve the monodromy. With $i_M$ expressed in the coordinates \eqref{eq:cup_semiorthogonal_decomp}, we see obtain a decomposition
  \begin{equation}
    \mathbf S_\phi \circ i_M \cong \mathrm{cone}(F_M \to F_-)
  \end{equation}
  with $F_-$ landing in the image of $\sigma^-$ and $F_M$ landing in the image of $i_M$. In particular, $F_M$ is left-orthogonal to $\cup_{\sigma^+}$, so Viterbo restriction $V$ to a neighborhood $U \cong F \times T^*[0,1]$ of $\sigma^- \cup \sigma^+$ gives an equivalence
  \begin{align*}
    (\cup_{\sigma^+}, \mathbf S_\phi \circ  i_M)^*\Delta_{\cF(M^{(2)})}
      &\cong (\cup_{\sigma^+}, F_-)^*\Delta_{\cF(M^{(2)})} \\
      &\cong (\cup_{\sigma^+}, V F_-)^*\Delta_{\cF(F \times T^*[0,1])} \\
      &\cong (\id_F, \Sigma^{-1}VF_-)^*\Delta_{\cF(F)} \\
      &= \Gamma^\dagger(\Sigma^{-1}VF_-) \\
      &\cong \Gamma^\dagger(\Sigma^{-1}V\mathbf S_\phi i_M).
  \end{align*}
  Here, $\Sigma$ is the stabilization equivalence $\cF(F) \cong \cF(F \times T^*[0,1])$, and we have used the fact that $V \circ i_M = 0$.
\end{proof}

We now perform the final computation of this section, which is in some sense the analogue of Lemma \ref{lem:Abouzaid-Auroux}:
\begin{lem}
  There is a commutative diagram of $A_\infty$ algebra maps
      \begin{equation}
      \begin{tikzcd}
        CF^*( \cap \left( \bR_+^n \times \bR_+^m   \right))  \ar[d] &  CF^*(\bR_+^n \times \bR_+^m) \ar[l]   \ar[d] \\
       \bZ[x_1, \ldots, x_n, y^{\pm}_1, \ldots, y_m^{\pm}]/\left( \prod x_j \right)  &   \bZ[x_1, \ldots, x_n, y^{\pm}_1, \ldots, y_m^{\pm}]  \ar[l] 
      \end{tikzcd}
    \end{equation}
    in which the vertical arrows are quasi-isomorphisms.
\end{lem}
\begin{proof}[Sketch of proof]
  By intrinsic formality, it suffices to compute the maps at the level of cohomology. An argument analogous to that of Proposition \ref{prop:kunneth_cup_compat} shows that the K\"unneth isomorphism commutes with Viterbo restriction as the cohomological level, so it suffices to show that we have a commutative diagram
 \begin{equation}
      \begin{tikzcd}
 HF^*( \cap \bR_+^n )  \ar[d] &         HF^*(\bR_+^n) \ar[l]   \ar[d] \\
     \bZ[x_1, \ldots, x_n]/\left( \prod x_j \right) &     \bZ[x_1, \ldots, x_n]  \ar[l]   
      \end{tikzcd}
    \end{equation}
    It is convenient at this stage to use the model of $H_{n-1}$ given as the $n-1$\st symmetric product of $\bC$ punctured at $n$ points, following \cite{LekiliPolishchuk2020}. In this model, we have an isomorphism
    \begin{equation}
            HF^*( \cap \bR_+^n )  \cong  \bZ[z_1, \ldots, z_n]/\left( \prod z_j \right), 
    \end{equation}
    moreover, each generator $z_j$ is localised near the $j$\th puncture. In particular, the inclusion of the product of small neighbourhoods of all but the $k$\th puncture yields the embedding of a Liouville subdomain
    \begin{equation}
      \Omega_k \to H_{n-1}      
    \end{equation}
    and the induced Viterbo restriction map is
    \begin{equation}
           \bZ[z_1, \ldots, z_n]/\left( \prod z_j \right) \to  \bZ[z^{\pm}_1, \ldots,z_{k-1}^{\pm} , z_k, z_{k+1}^{\pm}, \cdots,  z^{\pm}_n]/ z_k.
    \end{equation}
By considering a tropical degeneration of the pair of pants, as in \cite{Abouzaid2006}, we see that the composite embedding
    \begin{equation}
         \Omega_k \to H_{n-1} \to  (\bC^*)^{n}   
    \end{equation}
    is Liouville isotopic to an embedding
\begin{equation}
         \Omega_k \to (\bC^*)^{n-1} \to  (\bC^*)^{n}.
       \end{equation}
       This identifies the composite
       \begin{equation}
           HF^*( \bR_+^n ) \to     HF^*( \cap \bR_+^n ) \to    HF^*( \bR_+^{n-1} )    
         \end{equation}
         as the map taking $x_i$ to $z_i$. We conclude the desired result using the fact that the map
         \begin{equation}
             \bZ[z_1, \ldots, z_n]/\left( \prod z_j \right) \to  \bigoplus_{k} \bZ[z^{\pm}_1, \ldots, z^{\pm}_{k-1}, z_{k+1}^{\pm}, \ldots, z^{\pm}_n].          
         \end{equation}
is injective. 
\end{proof}

\appendix

  \section{Compatibility of K\"unneth maps with cup functors}
\label{sec:comp-k:unn-maps}

The purpose of this appendix is to prove the compatibility of K\"unneth isomorphisms with cup functors. To formulate the desired result, consider a Liouville sector $M_1$, a Liouville domain $M_2$, and embedded Lagrangians $L_i, L'_i \subset M_i$ which are invariant under the respective Liouville flows: 
\begin{prop}
  \label{prop:kunneth_cup_compat}
There is a cohomological K\"unneth isomorphism
  \begin{equation}
 HF^*(L_1, L'_1) \otimes  HF^*(L_2, L'_2)  \cong HF^*(L_1 \times L_2, L'_1 \times L'_2)
  \end{equation}
  which intertwines the cup functors in the sense that, if $\sigma_1$ is a boundary component of $M_1$, and $L'_1 = \cup \ell'_1$, then the following diagram commutes
  \begin{equation}
    \begin{tikzcd}
       HF^*(\ell_1, \ell'_1) \otimes  HF^*(L_2, L'_2)  \ar[d] \ar[r] &  HF^*(\ell_1 \times L_2, \ell'_1 \times L'_2) \ar[d] \\
       HF^*(L_1, L'_1) \otimes  HF^*(L_2, L'_2)  \ar[r] & HF^*(L_1 \times L_2, L'_1 \times L'_2).
    \end{tikzcd}
  \end{equation}
 \end{prop}

 There are three technical restrictions in the above theorem, which are all made for convenience: the restriction to cohomological information, to the setting where $M_2$ is a Liouville manifold, and to the case where all Lagrangians are globally invariant under the Liouville flow. The first allows us to avoid the technicalities of chain-level constructions, the second to avoid discussing Liouville sectors with corners, and the last ensures that products are conical at infinity.
 \begin{rem}
   We may replace the condition that the Lagrangians are globally conical by the condition that they be conical at infinity by applying the following trick from \cite{AbouzaidSeidel2010}: take the product of $M_i$ with $T^* S^1$, the product of $L_i$ and $L'_i$ with different cotangent fibres, and then equip the product $M_i \times T^* S^1$ with a Liouville structure for which these Lagrangians are globally conical. The Floer cohomology groups in $M_i \times T^* S^1$ decompose according to the winding number along the circle, and those of winding number $0$ agree with the Floer groups in $M_i$. Alternatively, we could take the product with the cotangent bundle of the interval, considered as a Liouville sector, at the expense of requiring additional explanation for how to address the appearance of corners.
 \end{rem}

\subsection{K\"unneth isomorphism}
\label{sec:kunneth-isomorphism}

We associate to a Liouville sector $M$ its horizontal completion
\begin{equation}
  \hat{M} \equiv M \cup \sigma \times T^* (-\infty,0] /  \sim,
\end{equation}
where $\sigma = \partial M / \bR$ is the stop of $M$.

We equip $\hat{M}$ with a Liouville form $\hat{\lambda}$ whose restriction to $\sigma \times T^* (-\infty,0] $ is the product of the Liouville form on $\sigma$ with a Liouville form on $T^* (-\infty, 0]$ whose flow is shown in Figure \ref{fig:Liouville-flow}. The boundary projection $\pi \colon \nu\partial M \to T^*[0,1)$ extends to a surjective projection
\begin{equation}
  \hat{\pi} \colon \widehat{ \nu \partial M } \to ( -\infty, 1 ).
\end{equation}

\begin{figure}[h]
  \centering
  \begin{tikzpicture}
    \draw (0,-2) -- (0,2);
    \foreach \i in {0, ..., 4}
      {
    \draw[->, thick] (-\i*.5,1) -- ++ (90:1);
    \draw[->, thick] (-\i*.5,-1) -- ++ (-90:1);
    \filldraw[black] (-\i*.5,0) circle (2pt);
  };
  \draw[->, thick] (-2.5,1) -- ++ (135:1);
  \draw[->, thick] (-2.5,0) -- ++ (180:1);
  \draw[->, thick] (-2.5,-1) -- ++ (-135:1);
    \foreach \i in {2, ..., 4}
      {
    \draw[->, thick] (-2,0) ++ (90+\i*30:2) -- ++ (90+\i*30:2);
  };

\end{tikzpicture} 
  \caption{The chosen Liouville flow on $T^* (-\infty, 0]$. The black dots represent the vanishing of the vector field on the subset of the $0$-section given by the interval $[-1,0]$.}
  \label{fig:Liouville-flow}
\end{figure}
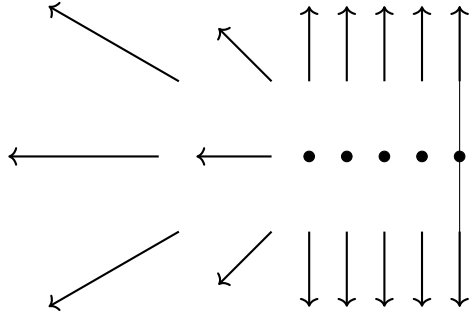

Since $\hat{M}$ is a complete Liouville manifold, we can construct a wrapped Fukaya category $\mathcal W(\hat M)$ using linear Hamiltonians and contact type almost complex structures \cite{AbouzaidSeidel2010} such that:
\begin{itemize}
  \item all Lagrangians are transverse to the locus $\hat{\pi}^{-1}\left( (-1, 1) \right)$ living over the zero section, and
  \item the Hamiltonian vector field $X_H$ of the wrapping Hamiltonians is tangent to each fiber $\hat{\pi}^{-1}(q)$ for $q \in (-1, 1)$, and
  \item The almost complex structure fixes each of the above fibers.
\end{itemize}

The wrapped Fukaya category of the sector $M$ is defined to be the non-full subcategory $\mathcal W(M) \subset \mathcal W(\hat{M})$, where
\begin{itemize}
  \item ahe Lagrangians are required to live in the interior of $M$, and
  \item the morphisms are generated by Hamiltonian chords supported in $M$.
\end{itemize}
That this is a subcategory follows from positivity of intersections of $H$-perturbed holomorphic disks with the fiber $\hat{\pi}^{-1}(0)$. This is the linear version of the category $\mathcal W_\mathrm{quad}(M)$ in \cite{Sylvan2019a}.

\subsubsection{Hamiltonians on product sectors}

We are interested in the situation where $M$ is a product $M_1 \times M_2$, where we would like to identify on the nose the Fukaya algebras associated to split and cylindrical Floer data. 

Fix a bump function $a \colon \bR \to [0,1]$ which vanishes for $r \le 2$, equals $1$ for $r \in [\beta_-, \beta_+]$, and vanishes again for $r$ sufficiently large. We will construct Hamiltonians on $M_1 \times M_2$ which are linear on a large cylinder and split outside a compact set by rescaling via the function $h(r) = a(r) \log(r)$.

\pgfdeclarelayer{bg}    
\pgfsetlayers{bg,main}  

\begin{figure}[h]
  \centering
  \begin{tikzpicture}
\begin{axis}[
  axis x line=left, axis y line=left, ylabel style={rotate=270},
  ymin=0, ymax=4, ytick={0,2}, ylabel=$H_2$,
  xmin=0, xmax=4, xtick={0,2}, xlabel=$H_1$
]
\draw[thin] (axis cs: 2,0) -- (axis cs: 2,2) -- (axis cs: 0,2);
\draw[thin, blue] (axis cs: 0,0) -- (axis cs: 4,1);
\draw[thin, blue] (axis cs: 0,0) -- (axis cs: 1,4);
\draw[black, fill=gray] (axis cs: 0,0) --(axis cs: 0,0.2) -- (axis cs: 0.2,0.2) -- (axis cs: 0.2,0) --(axis cs: 0,0);
\begin{pgfonlayer}{bg}
\fill[yellow, opacity=50] (axis cs: 2,0) -- (axis cs: 2,.5) -- (axis cs: 4,1) -- (axis cs: 4,0) -- cycle ;
\fill[yellow, opacity=50] (axis cs: 0,2) -- (axis cs: .5,2) -- (axis cs: 1,4) -- (axis cs: 0,4) -- cycle ;  
\end{pgfonlayer}
\end{axis}
\end{tikzpicture} 
\caption{The decomposition of $\hat{M}_1 \times \hat{M}_2$}
  \label{fig:decompose_product}
\end{figure}
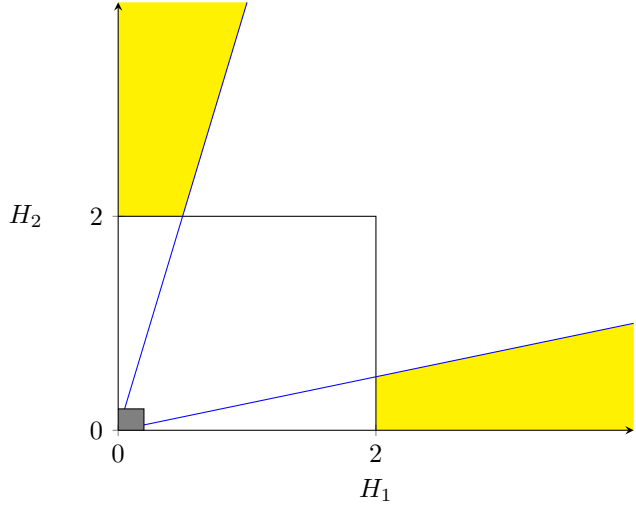

Concretely, let $H_1$ and $H_2$ be linear exhausting positive Hamiltonians on $M_1$ and $M_2$, respectively, so that
\begin{equation} \label{eq:bound_derivative_Hi-Liouville}
    0 \leq Z_iH_i \leq H_i,
\end{equation}
where the second inequality is an equality whenever $\varepsilon \leq H_i$ (i.e. $H_i$ is linear outside a compact set). Define
\begin{equation}
  \label{eq:product_Hamiltonian}
  H_{12}(x_1, x_2) =
  e^{h(H_1)} \cdot H_2\left(\phi_{Z_2}^{-h(H_1)}x_2\right)
  +
  e^{h(H_2)} \cdot H_1\left(\phi_{Z_1}^{-h(H_2)}x_1\right).
\end{equation}
This construction satisfies the following key properties:
\begin{enumerate}
  \item Provided $\varepsilon < 1$ and $\beta_+ - \beta_-$ is sufficiently large relative to $\frac1\varepsilon$, $H_{12}$ is linear along $H_{12}^{-1}(\beta_+)$.
  \item The region where $H_{12}$ fails to agree with $H_1 + H_2$ is compact, and is disjoint from the locus where $\varepsilon < \frac{H_2}{H_1} < \frac1\varepsilon$.
\end{enumerate}
Henceforth, we will assume $\varepsilon < 1$, and in fact we will shrink $\varepsilon$ to further improve the behavior of $H_{12}$.

Our goal is to choose $H_1$ and $H_2$ so that all chords of $H_{12}$ live in the region $H_{12} < \beta_+$. To that end, note that the region where $H_{12}$ is not split decomposes into two disjoint regions, where $H_1 \ge \max\left\{2, \frac{H_2}\varepsilon\right\}$ and where $H_2 \ge \max\left\{2, \frac{H_1}\varepsilon\right\}$, see Figure \ref{fig:decompose_product}. These two regions are essentially identical, so we will focus on the first.

Let us compute $X_{H_{12}}$. First, the $M_1$ component of $dH_{12}$ is given by
\begin{multline}
  \left[ e^{h(H_1)} \cdot \left( H_2 - Z_2H_2 \right)\left(\phi_{Z_2}^{-h(H_1)}x_2\right) \cdot h'(H_1) \right]dH_1 \\
  + e^{h(H_2)} \cdot dH_1 \circ d\phi_Z^{-h(H_2)}
\end{multline}
and similarly with the $M_2$ component. In the region $H_1 \ge \max\left\{ 2, \frac{H_2}\varepsilon \right\}$, this simplifies to
\begin{equation}
  \label{eq:d_split_Hamiltonian}
  \left[ 1 + e^{h(H_1)}h'(H_1) \cdot \left( H_2 - Z_2H_2 \right)\left(\phi_{Z_2}^{-h(H_1)}x_2\right) \right] dH_1,
\end{equation}
while the $M_2$ component simplifies to
\begin{equation}
  e^{h(H_1)} \cdot dH_2 \circ d\phi_Z^{-h(H_1)}.
\end{equation}

Let us examine Equaction \eqref{eq:d_split_Hamiltonian}. The coefficient $e^{h(H_1)}h'(H_1)$ is globally bounded independently of $\varepsilon$. By shrinking $\varepsilon \rightsquigarrow \varepsilon'$ and correspondingly rescaling $H_i \rightsquigarrow \frac{\varepsilon'}{\varepsilon}H_i \circ \phi_{Z_i}^{\log\frac{\varepsilon}{\varepsilon'}}$, we can make $\lVert H_2-Z_2H_2\rVert_{C^0}$ as small as we like, so that Equation \eqref{eq:d_split_Hamiltonian} approaches $dH_1$. Since $X_{H_1}$ has no chords in the linear region $H_1 > \varepsilon$, neither will $H_{12}$ once $\varepsilon$ is small enough.

After further shrinking $\varepsilon$, the same argument shows that we can arrange that $mX_{H_{12}}$ also has no chords in the region where $H_{12} \ge \beta_+$. We have proved

\begin{lemma}
  \label{lem:product_ham}
  Let $H_1$ and $H_2$ be linear exhausting positive Hamiltonians on $M_1$ and $M_2$ which have no integer length chords from $\partial^\infty L_i$ to  $\partial^\infty L'_i$, and which satisfy Equation~\eqref{eq:bound_derivative_Hi-Liouville}. Let $m \in \bN$. Then, for any sufficiently small $\varepsilon > 0$, there are linear exhausting positive Hamiltonians $H_1^{(m)}$ on $M_1$ and $H_2^{(m)}$ on $M_2$, together with a positive exhausting (non-linear) Hamiltonian $H_{12}^{(m)}$ on $M_1 \times M_2$ such that:
  \begin{enumerate}
    \item $H_i^{(m)}$ is a rescaling of $H_i$ which is linear in the region where it is bigger than $\varepsilon$, and satisfies $H_i^{(m)} \leq H_i $ everywhere.
    \item $H_{12}^{(m)}$ is linear along ${H_{12}^{(m)}}^{-1}(\beta_+)$.
    \item $H_{12}^{(m)}$ agrees with $H_1^{(m)} + H_2^{(m)}$ except on the regions $\max\left\{2, \frac{H_2^{(m)}}\varepsilon\right\} \le H_1^{(m)} \le \beta_{\max}$ and $\max\left\{2, \frac{H_1^{(m)}}\varepsilon\right\} \le H_2^{(m)} \le \beta_{\max}$. Here, $\beta_{\max}$ is a value above which the bump function $a$ vanishes.
    \item For $j = 1, \dotsc, m$, all time $1$ chords of $jX_{H_{12}^{(m)}}$ from $L_1 \times L_2$ to  $L'_1 \times L'_2$ are split. In particular, they lie in the region where both $H_1^{(m)}$ and $H_2^{(m)}$ are less than $\varepsilon$.
  \end{enumerate}
  \qed
\end{lemma}

Fix once and for all the model Hamiltonians $H_i$. Consider a sequence $(m, \varepsilon_m)_{m\in \bN}$ and corresponding Hamiltonians $H_1^{(m)}$, $H_2^{(m)}$, and $H_{12}^{(m)}$.
We consider the homotopy-commutative diagram
\begin{equation}
  \label{diag:wrapping_hams}
  \begin{tikzcd}[column sep=5]
    CF^*(L_1 \times L_2,L'_1 \times L'_2 ; H_{12}^{(1)}) \ar[r, "\cong"] & CF^*(L_1 \times L_2,L'_1 \times L'_2 ; H_{12}^{(2)}) \ar[r, "\cong"] \ar[d] & CF^*(L_1 \times L_2,L'_1 \times L'_2 ; H_{12}^{(3)}) \ar[r] \ar[d] & \cdots \ar[d] \\
                                                       & CF^*(L_1 \times L_2,L'_1 \times L'_2 ; 2H_{12}^{(2)}) \ar[r, "\cong"] & CF^*(L_1 \times L_2,L'_1 \times L'_2 ; 2H_{12}^{(3)}) \ar[r] \ar[d] & \cdots \ar[d] \\
                                                       & & CF^*(L_1 \times L_2,L'_1 \times L'_2 ; 3H_{12}^{(3)}) \ar[r]
                                                      & \cdots \ar[d] \\
                                                       & & & \cdots 
  \end{tikzcd}
\end{equation}
in which the vertical maps are continuation maps associated to increasing slope \cite{AbouzaidSeidel2010}. The horizontal maps are constructed by the homotopy method \cite{FukayaOhOhtaOno2009}:   since all chords of $jH_{12}^{(m)}$ are split, they are products of $X_{H_1^{(m)}}$-chords and $X_{H_2^{(m)}}$-chords of length $j$, which are rescalings of length $j$ chords of $X_{H_1}$ and $X_{H_2}$. In particular, further rescaling $H_1$ and $H_2$ by shrinking $\varepsilon$ doesn't change the length $j$ chords. The formula in Equation \eqref{eq:product_Hamiltonian} thus gives a family $H_{12}^t$ which interpolates between $H_{12}^{(m)}$ and $H_{12}^{(m+1)}$ with no birth/death of chords. In particular, we obtain canonical identifications of the generators of $CF^*(L_1 \times L_2,L'_1 \times L'_2; H_{12}^t)$ for varying $t$.   The homotopy method quasi-isomorphisms are then obtained by counting time-ordered arrangements of exceptional index $0$ strips, where the generators at different times are identified as above (see Figure \ref{fig:Homotopy_method}). For homotopy-commutativity, one similarly counts time-ordered arrangements of index $-1$ in which one of the strips is an exceptional continuation map, exactly as in the construction of the linear term of the Viterbo restriction map in \cite{AbouzaidSeidel2010}.

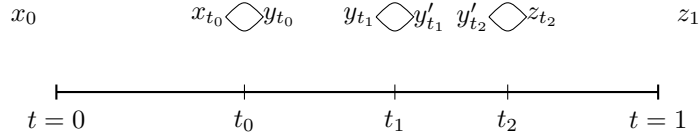
\begin{figure}[h]
  \centering
  \begin{tikzpicture}
    \draw[thick] (-4,0) -- (4,0);
    \draw[thick] (-4,-.1) -- (-4,.1);
    \draw[thick] (4,-.1) -- (4,.1);
      \node[label=below:{$t=0$}] at (-4,0) {};
      \node[label=below:{$t=1$}] at (4,0) {};

   \begin{scope}[shift={(0,1)}]
     \node[label=left:{$x_0$}] at (-4,0) {};
     \begin{scope}[shift={(-1.5,0)}]
         \node[label=left:{$x_{t_0}$}] at (0,0) {};
     \draw (-.25,0) .. controls (0,.25) .. (.25,0) .. controls (0,-.25) .. (-.25,0);
     \node[label=right:{$y_{t_0}$}] at (0,0) {};
     \begin{scope}[shift={(0,-1)}]
        \draw[thin] (0,-.1) -- (0,.1);
   \node[label=below:{$t_0$}] at (0,0) {};
     \end{scope}
     \end{scope}
   \begin{scope}[shift={(.5,0)}]
         \node[label=left:{$y_{t_1}$}] at (0,0) {};
     \draw (-.25,0) .. controls (0,.25) .. (.25,0) .. controls (0,-.25) .. (-.25,0);
     \node[label=right:{$y'_{t_1}$}] at (0,0) {};
      \begin{scope}[shift={(0,-1)}]
        \draw[thin] (0,-.1) -- (0,.1);
   \node[label=below:{$t_1$}] at (0,0) {};
     \end{scope}
     \end{scope}
     \begin{scope}[shift={(2,0)}]
         \node[label=left:{$y'_{t_2}$}] at (0,0) {};
     \draw (-.25,0) .. controls (0,.25) .. (.25,0) .. controls (0,-.25) .. (-.25,0);
     \node[label=right:{$z_{t_2}$}] at (0,0) {};
       \begin{scope}[shift={(0,-1)}]
        \draw[thin] (0,-.1) -- (0,.1);
   \node[label=below:{$t_2$}] at (0,0) {};
     \end{scope}
   \end{scope}
   
          \node[label=right:{$z_1$}] at (4,0) {};
   \end{scope}
\end{tikzpicture} 
  \caption{A representation of the time-ordered strip count in the homotopy method, consisting of three rigid solutions to the Floer equation, occuring at times $t_0 < t_1 < t_2$. The labels $x_t$, $y_t$, $y'_t$, and $z_t$ are those for the time-$1$ chords of $H_t$.}
  \label{fig:Homotopy_method}
\end{figure}

\begin{lem}
  The colimit of the homology groups in Diagram \eqref{diag:wrapping_hams} computes wrapped Floer cohomology in the completion of $M_1 \times M_2$.
\end{lem}
\begin{proof}
   Write $\tilde H_{12}^{(m)}$ for the linear extrapolation of $H_{12}^{(m)}$, which one can construct by replacing $\beta_+$ with infinity. Note that, for $j = 1, \dotsc, m$, $CF^*(L_1 \times L_2; jH_{12}^{(m)})$ and $CF^*(L_1 \times L_2; \tilde jH_{12}^{(m)})$ are canonically isomorphic. Indeed, they have the same chords, and by the integrated maximum principle \cite{AbouzaidSeidel2010} along ${H_{12}^{(m)}}^{-1}(\beta_+)$, all holomorphic disks live in the region where the Hamiltonians agree.

We begin by observing that
\begin{equation}
  H_{12}^{(m)} \ge \tilde H_{12}^{(m)} \ge \max(H_1^{(m)}, H_2^{(m)})
\end{equation}
globally. The right hand side is a (non-smooth) linear exhausting positive Hamiltonian which is eventually $m$-independent outside an arbitrarily small neighborhood of the skeleton of $M_1 \times M_2$. In particular, $m\tilde H_{12}^{(m)}$ has slope tending to infinity with $m$, so
\begin{equation} \label{eq:colimit_HW_completed}
  \lim_\rightarrow HF^*(L_1\times L_2,L'_1\times L'_2 ; m_i\tilde H_{12}^{(m_i)})
\end{equation}
computes the wrapped Floer cohomology of $L_1\times L_2$ and $L'_1\times L'_2$, where the values of $m_i$ appearing in the colimit are chosen so that $m_{i+1}H_{12}^{(m_{i+1})} > m_i H_{12}^{(m_i)}$ outside a compact set.

It now suffices to compare these continuation maps with the maps obtained from Diagram \eqref{diag:wrapping_hams}. First we use the integrated maximum principle again to see that the colimit in Equation \eqref{eq:colimit_HW_completed} is isomorphic to the colimit
\begin{equation} \label{eq:colimit_HW_split}
  \lim_\rightarrow HF^*(L_1\times L_2,L'_1\times L'_2 ; m_i H_{12}^{(m_i)}),
\end{equation}
where we recall that $H_{12}^{(m_i)}$ is split at infinity. Since the vertical maps are continuation maps, their compatibility with composition reduces the problem to showing that the (inverse) of the horizontal maps agrees on cohomology with continuation maps, which are well-defined because the functions $H_{12}^{t}$ decrease with $t$ as a consequence of Equation \eqref{eq:bound_derivative_Hi-Liouville}. 

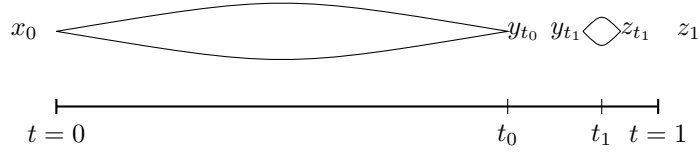
\begin{figure}[h]
  \centering
  \begin{tikzpicture}
    \draw[thick] (-4,0) -- (4,0);
    \draw[thick] (-4,-.1) -- (-4,.1);
    \draw[thick] (4,-.1) -- (4,.1);
      \node[label=below:{$t=0$}] at (-4,0) {};
      \node[label=below:{$t=1$}] at (4,0) {};

   \begin{scope}[shift={(0,1)}]
     \node[label=left:{$x_0$}] at (-4,0) {};
        \draw (-4,0) .. controls (-1,.5) .. (2,0) .. controls (-1,-.5) .. (-4,0);
     \node[label=right:{$y_{t_0}$}] at (1.75,0) {};
     \begin{scope}[shift={(0,-1)}]
        \draw[thin] (2,-.1) -- (2,.1);
   \node[label=below:{$t_0$}] at (2,0) {};
     \end{scope}

   \begin{scope}[shift={(3.25,0)}]
         \node[label=left:{$y_{t_1}$}] at (0,0) {};
     \draw (-.25,0) .. controls (0,.25) .. (.25,0) .. controls (0,-.25) .. (-.25,0);
     \node[label=right:{$z_{t_1}$}] at (0,0) {};
      \begin{scope}[shift={(0,-1)}]
        \draw[thin] (0,-.1) -- (0,.1);
   \node[label=below:{$t_1$}] at (0,0) {};
     \end{scope}
     \end{scope}
          \node[label=right:{$z_1$}] at (4,0) {};
   \end{scope}

\end{tikzpicture} 
  \caption{A representation of the time-ordered strip count in the comparison between the homotopy method and the continuation map, consisting of a rigid solutions to the continuation map from $t=0$ to $t=t_0$, followed by a solution to the Floer equation, occuring at some time $t_0 < t_1$.}
  \label{fig:Homotopy_method_continuation}
\end{figure}

This is achieved by using the composition of continuation maps from $mH_{12}^{(n)}$ to  $mH_{12}^{(t)}$, followed by homotopy from $mH_{12}^{(t)}$  to $mH_{12}^{(m)}$, associated to $m \leq t \leq n$ (see Figure \ref{fig:Homotopy_method_continuation}). The moduli space of solutions to these equations is compact by an application of the maximum principle in the two factors, and the associated parametrised moduli space defines the desired homotopy.

\end{proof}

Finally, we complete this appendix by proving it main result:
\begin{proof}[Proof of Proposition \ref{prop:kunneth_cup_compat}]
  In the setting of Hamiltonian wrapping, the cup functor $\cup_\sigma \colon \cF(\sigma) \to \cF(M)$ sends a Lagrangian $\ell \subset \sigma$ to its product with a fibre $\ell \times T_{\frac12}^*[0,1) \subset M$. At the level of morphisms, it sends $CW^*(\ell, \ell')$ to the piece of $CW^*(\ell\times T_{\frac12}^*[0,1), \ell'\times T_{\frac12}^*[0,1))$ living along the zero section of $[0,1)$ \cite[Construction 2.18]{Sylvan2019a}. The assumptions on our wrapping Hamiltonians imply that this piece is a subalgebra and that it is isomorphic to $CW^*(\ell, \ell')$ itself.

  To prove the proposition, it is enough to check it for finite slope. Since all of the Hamiltonians $H_{12}$, $\tilde H_{12}$, and $H_1+H_2$ preserve the fibre over $(q=\frac12, p=0)$ near the stop of $M_1\times M_2$, we obtain a diagram
  \begin{equation}
  \begin{tikzcd}
    CF^*(\ell_1 \times L_2 , \ell'_1 \times L'_2 ; j\tilde H_{12}^{(m)}) \ar[d, "="]\ar[r, "\cup"] & CF^*(\cup\ell_1 \times L_2, \cup \ell'_1 \times L'_2 ; j\tilde H_{12}^{(m)}) \ar[d, "="]  \\
    CF^*(\ell_1 \times L_2, \ell'_1 \times L'_2 ; jH_{12}^{(m)}) \ar[d, "\cong"] \ar[r, "\cup"] &   CF^*(\cup\ell_1 \times L_2, \cup \ell'_1 \times L'_2; jH_{12}^{(m)}) \ar[d, "\cong"] \\
  CF^*(\ell_1 \times L_2, \ell'_1 \times L'_2; jH_1^{(m)} + jH_2^{(m)}) \ar[r, "\cup"]  & CF^*(\cup\ell_1 \times L_2, \cup \ell'_1 \times L'_2; jH_1^{(m)} + jH_2^{(m)})
  \end{tikzcd}
  \end{equation}
  where we have abused notation and referred to the Hamiltonians in the fibre and the total space in the same way. The top vertical arrows are the canonical isomorphisms above, while the bottom vertical arrows are continuation quasi-isomorphisms coming from the fact that $H_{12}^{(m)} - (H_1^{(m)} + H_2^{(m)})$ is compactly supported. The top square commutes because all arrows are identifications of (sub)complexes. The bottom square commutes because the strips contributing to the continuation map which start in the fibre must stay in the fibre, for the same reason that the generators living in the fibre form a subcomplex.

  Since the right-hand cup functor agrees with $\cup_{\sigma_1} \times \id_{\cF(M_2)}$, the result follows.
\end{proof}

\bibliographystyle{plain}
\bibliography{large-bib}

\def\cprime{$'$}
\begin{thebibliography}{10}

\bibitem{AbouzaidSeidel2010}
M.~Abouzaid and P.~Seidel.
\newblock An open string analogue of {V}iterbo functoriality.
\newblock {\em Geom. Topol.}, 14:627--718, 2010.

\bibitem{Abouzaid2006}
Mohammed Abouzaid.
\newblock Homogeneous coordinate rings and mirror symmetry for toric varieties.
\newblock {\em Geom. Topol.}, 10:1097--1156, 2006.
\newblock [Paging previously given as 1097--1157].

\bibitem{Abouzaid2008}
Mohammed Abouzaid.
\newblock On the {F}ukaya categories of higher genus surfaces.
\newblock {\em Adv. Math.}, 217(3):1192--1235, 2008.

\bibitem{Abouzaid2011}
Mohammed Abouzaid.
\newblock A cotangent fibre generates the {F}ukaya category.
\newblock {\em Adv. Math.}, 228(2):894--939, 2011.

\bibitem{Abouzaid2012a}
Mohammed Abouzaid.
\newblock Nearby {L}agrangians with vanishing {M}aslov class are homotopy
  equivalent.
\newblock {\em Inventiones mathematicae}, 189:251--313, 2012.

\bibitem{AbouzaidAuroux2021}
Mohammed Abouzaid and Denis Auroux.
\newblock Homological mirror symmetry for hypersurfaces in
  {$(\mathbb{C}^*)^n$}, 2021.
\newblock in preparation.

\bibitem{AbouzaidAurouxEfimovKatzarkovOrlov2013}
Mohammed Abouzaid, Denis Auroux, Alexander~I. Efimov, Ludmil Katzarkov, and
  Dmitri Orlov.
\newblock Homological mirror symmetry for punctured spheres.
\newblock {\em J. Amer. Math. Soc.}, 26(4):1051--1083, 2013.

\bibitem{AbouzaidAurouxKatzarkov2016}
Mohammed Abouzaid, Denis Auroux, and Ludmil Katzarkov.
\newblock Lagrangian fibrations on blowups of toric varieties and mirror
  symmetry for hypersurfaces.
\newblock {\em Publ. Math. Inst. Hautes \'{E}tudes Sci.}, 123:199--282, 2016.

\bibitem{AbouzaidSmith2019}
Mohammed Abouzaid and Ivan Smith.
\newblock Khovanov homology from {F}loer cohomology.
\newblock {\em J. Amer. Math. Soc.}, 32(1):1--79, 2019.

\bibitem{ChanPomerleanoUeda2016}
Kwokwai Chan, Daniel Pomerleano, and Kazushi Ueda.
\newblock Lagrangian torus fibrations and homological mirror symmetry for the
  conifold.
\newblock {\em Comm. Math. Phys.}, 341(1):135--178, 2016.

\bibitem{ChantraineDimitroglouGhigginiGolovko2019}
Baptiste Chantraine, Georgios~Dimitroglou Rizell, Paolo Ghiggini, and Roman
  Golovko.
\newblock Geometric generation of the wrapped fukaya category of weinstein
  manifolds and sectors, 2019.

\bibitem{EliashbergGromov1991}
Yakov Eliashberg and Mikhael Gromov.
\newblock Convex symplectic manifolds.
\newblock In {\em Several complex variables and complex geometry, {P}art 2
  ({S}anta {C}ruz, {CA}, 1989)}, volume~52 of {\em Proc. Sympos. Pure Math.},
  pages 135--162. Amer. Math. Soc., Providence, RI, 1991.

\bibitem{FukayaOhOhtaOno2009}
Kenji Fukaya, Yong-Geun Oh, Hiroshi Ohta, and Kaoru Ono.
\newblock {\em Lagrangian intersection {F}loer theory: anomaly and obstruction.
  {P}art {I}}, volume~46 of {\em AMS/IP Studies in Advanced Mathematics}.
\newblock American Mathematical Society, Providence, RI; International Press,
  Somerville, MA, 2009.

\bibitem{Gammage2021}
Benjamin Gammage.
\newblock Local mirror symmetry via syz, 2021.

\bibitem{GammageLe2021}
Benjamin Gammage and Ian Le.
\newblock Mirror symmetry for truncated cluster varieties, 2021.

\bibitem{GammageShende2017}
Benjamin Gammage and Vivek Shende.
\newblock Mirror symmetry for very affine hypersurfaces, 2021.

\bibitem{GanatraPardonShende2019}
Sheel Ganatra, John Pardon, and Vivek Shende.
\newblock Sectorial descent for wrapped fukaya categories, 2019.

\bibitem{GanatraPomerleano2021}
Sheel Ganatra and Daniel Pomerleano.
\newblock A log {PSS} morphism with applications to {L}agrangian embeddings.
\newblock {\em J. Topol.}, 14(1):291--368, 2021.

\bibitem{HanlonHicks2021}
Andrew Hanlon and Jeff Hicks.
\newblock Functoriality and homological mirror symmetry for toric varieties,
  2021.

\bibitem{Kuznetsov2016}
Alexander Kuznetsov.
\newblock Derived categories view on rationality problems.
\newblock In {\em Rationality problems in algebraic geometry}, volume 2172 of
  {\em Lecture Notes in Math.}, pages 67--104. Springer, Cham, 2016.

\bibitem{Lazarev2019}
Oleg Lazarev.
\newblock Geometric and algebraic presentations of weinstein domains, 2019.

\bibitem{LekiliPolishchuk2020}
Yank\i Lekili and Alexander Polishchuk.
\newblock Homological mirror symmetry for higher-dimensional pairs of pants.
\newblock {\em Compos. Math.}, 156(7):1310--1347, 2020.

\bibitem{Mikhalkin2004}
Grigory Mikhalkin.
\newblock Decomposition into pairs-of-pants for complex algebraic
  hypersurfaces.
\newblock {\em Topology}, 43(5):1035--1065, 2004.

\bibitem{Orlov1992}
D.~Orlov.
\newblock Projective bundles, monoidal transformations, and derived categories
  of coherent sheaves.
\newblock {\em Izv. Ross. Akad. Nauk Ser. Mat.}, 56:852--862, 1992.

\bibitem{Pomerleano2021}
Daniel Pomerleano.
\newblock Intrinsic mirror symmetry and categorical crepant resolutions, 2021.

\bibitem{Seidel2013}
Paul Seidel.
\newblock Notes on categorical dynamics and symplectic topology.

\bibitem{Seidel2002}
Paul Seidel.
\newblock Fukaya categories and deformations.
\newblock In {\em Proceedings of the {I}nternational {C}ongress of
  {M}athematicians, {V}ol. {II} ({B}eijing, 2002)}, pages 351--360. Higher Ed.
  Press, Beijing, 2002.

\bibitem{Seidel2008}
Paul Seidel.
\newblock {\em Fukaya categories and {P}icard-{L}efschetz theory}.
\newblock Zurich Lectures in Advanced Mathematics. European Mathematical
  Society (EMS), Z\"{u}rich, 2008.

\bibitem{Seidel2011}
Paul Seidel.
\newblock Homological mirror symmetry for the genus two curve.
\newblock {\em J. Algebraic Geom.}, 20(4):727--769, 2011.

\bibitem{SeidelSmith2005}
Paul Seidel and Ivan Smith.
\newblock The symplectic topology of {R}amanujam's surface.
\newblock {\em Comment. Math. Helv.}, 80(4):859--881, 2005.

\bibitem{Sheridan2011}
N.~Sheridan.
\newblock On the homological mirror symmetry conjecture for pairs of pants.
\newblock {\em J. Differential Geom.}, 89:271--367, 2011.

\bibitem{Sylvan2019}
Zachary Sylvan.
\newblock On partially wrapped {F}ukaya categories.
\newblock {\em J. Topol.}, 12(2):372--441, 2019.

\bibitem{Sylvan2019a}
Zachary Sylvan.
\newblock Orlov and viterbo functors in partially wrapped fukaya categories,
  2019.

\end{thebibliography}

\end{document}